\documentclass[11pt]{article}
\usepackage{amssymb,amsmath,amsfonts,amsthm,array,bm,bbm,color,graphicx,hyperref}

\hypersetup{colorlinks=true, linkcolor=red, anchorcolor=blue, citecolor=blue}
\providecommand{\U}[1]{\protect\rule{.1in}{.1in}}

\numberwithin{equation}{section}

\newtheorem{theorem}{Theorem}[section]
\newtheorem{lemma}[theorem]{Lemma}
\newtheorem{corollary}[theorem]{Corollary}
\newtheorem{proposition}[theorem]{Proposition}
\newtheorem{remark}[theorem]{Remark}

\newtheorem{definition}[theorem]{Definition}

\def\<{\langle}
\def\>{\rangle}
\def\d{{\rm d}}

\def\E{\mathbb{E}}

\def\P{\mathbb{P}}
\def\R{\mathbb{R}}

\def\F{\mathcal{F}}
\def\p{\partial}
\def\eps{\varepsilon}

\def\law{\operatorname{Law}}

\definecolor{dark}{rgb}{0.0, 0.42, 0.0}

\begin{document}

\title{Interacting particle approximation of cross-diffusion systems}
\author{José Antonio Carrillo\footnote{Email: carrillo@maths.ox.ac.uk. Mathematical Institute, University of Oxford, Oxford OX2 6GG, UK.} \,   Shuchen Guo\footnote{Email: guo@maths.ox.ac.uk. Mathematical Institute, University of Oxford, Oxford OX2 6GG, UK.} }

\maketitle

\vspace{-20pt}

\begin{abstract}
We prove the existence of weak solutions of a class of multi-species cross-diffusion systems as well as the propagation of chaos result by means of nonlocal approximation of the nonlinear diffusion terms, coupling methods and compactness arguments. We also prove the uniqueness under further structural assumption on the mobilities by combining the uniqueness argument for viscous porous medium equations and linear Fokker-Planck equations. We show that these equations capture the macroscopic behavior of stochastic interacting particle systems if the localisation parameter is chosen logarithmically with respect to the number of particles. 
\end{abstract}



\section{Introduction}
Multi-species cross-diffusion models are systems of coupled equations which describe the evolution of densities of $n$ different species $(n\geq 2)$. The solution of the cross-diffusion system is a vector-valued function 
$
\rho=(\rho_1,\ldots,\rho_n)
$
defined on  $\R^d$,
where $\rho_k$ is the 
density of $k$-th $(k=1,2,\ldots,n)$ species. 
We consider a class of cross-diffusion systems on $\R^d$ as follows
\begin{equation}\label{cross-diffusion1}
 \begin{cases}
  \displaystyle\p_t \rho_k-b_k\nabla\cdot \big(\rho_{k}\nabla P(\rho)\big)=\sigma\Delta \rho_{k},\\
  \rho_{k}(0)=\rho_{k,0}, \quad \rho_{k,0}\in  L^1\cap L^\infty(\R^d),
  \end{cases} k=1,2,\ldots,n,
\end{equation}
where the parameter $b_k>0$ denotes the mobility, $\sigma>0$ denotes the diffusion coefficient, and the motion is driven by pressure gradients through Darcy's law, for   $a_k>0$,
$$
P(\rho)=\left(\sum_{k=1}^{n}a_k\rho_k\right)^{m-1},\quad m\geq 2.
$$ 
The system can be written in gradient flow structure as 
\begin{equation}\label{x}
\p_t \rho_k=\frac{b_k}{a_k}\nabla\cdot \left(\rho_k\nabla \frac{\delta\mathcal{A}}{\delta\rho_k}\right),
\end{equation}
where the energy functional $\mathcal{A}$ is given by 
\begin{equation}\label{energy functional}
\mathcal{A}[\rho_1,\ldots,\rho_n]=\frac{1}{m}\int_{\R^d}\left(\sum_{k=1}^{n}a_k\rho_k\right)^{m}\d x + \sum_{k=1}^{n}\frac{a_k}{b_k}\sigma \int_{\R^d} \rho_k\log \rho_k\, \d x.    
\end{equation}
The energy functional defined above can be regularised at a formal level,
\begin{equation}\label{regular A}
\mathcal{A}_\eps[\rho_1^\eps,\ldots,\rho_n^\eps]=\frac{1}{m}\int_{\R^d}\left(\sum_{k=1}^{n}a_kV^\eps\ast\rho_k^\eps\right)^{m}\d x + \sum_{k=1}^{n}\frac{a_k}{b_k}\sigma \int_{\R^d} \rho_k^\eps\log \rho_k^\eps \,\d x,
\end{equation}
where we denote the variable of the regularised functional as $\rho^\eps=(\rho_1^\eps,\ldots,\rho_n^\eps)$, and mollifer $V^\eps$ which is obtained from nonnegative even function $V\in C^\infty_c(\R^d)$ with $\int_{\R^d}V(x)\d x=1$, by scaling
$V^{\eps}(x):=\eps^{-d} V(x/\eps  )  $. This regularised functional leads to the nonlocal equation 
\begin{equation}\label{eps cross}
\p_t \rho_k^{\eps}=\nabla\cdot \Big(b_k\rho_k^{\eps}\nabla V^{\eps} \ast \big(\sum_{l=1}^{n}a_lV^{\eps} \ast  \rho_l^{\eps}\big)^{m-1}\Big)+\sigma \Delta \rho_k^{\eps},    
\end{equation}
with initial data $\rho_k^\eps(0)=\rho_{k,0}$,  which plays an important role below as an intermediate system between the cross-diffusion system and its particle approximation. 

From a physical perspective, each subpopulation consists of a large number of interacting particles, which can represent molecules, cells, individuals, and so on depending on the application. Our motivation is to derive the cross-diffusion system \eqref{cross-diffusion1} from stochastic many-particle systems. For the sake of notational simplicity, we take the same numbers of particles in each species as $N\in\mathbb{N}$. Let $\big(\Omega,\mathcal{F}, (\mathcal{F}_t)_{t\geq 0},\P\big)$ be a filtered probability space, for any $k=1,\ldots, n$, $\big(\xi_k^i\big)_{i\geq 1}$ is a sequence of i.i.d random variables on $\R^d$ with the common law $\rho_{k,0}$, $(B_k^i)_{i\geq 1}$ are i.i.d $d$-dimension $\F_t$-Brownian motions that are independent of $\xi_k^i$. The dynamics of particle system of $k$-th species is described by the following SDEs, for $k=1,\ldots,n$ and $i=1,\cdots,N$ that 
\begin{equation}\label{cross-X}
\begin{cases}
 \d X_{k}^{i,\eps}(t)  =-b_k\Big[\nabla V^{\eps} \ast \Big(\displaystyle\sum_{l=1}^{n}\frac{a_{l}}{N}\sum_{j=1}^NV^{\eps} \big(\cdot-X_l^{j,\eps}(t)\big)\Big)^{m-1}\Big]\big(X_{k}^{i,\eps}(t) \big) \d t+\sqrt{2\sigma }\d B_k^{i}(t),   \\
  X_{k}^{i,\eps}(0)=\xi_k^i,
\end{cases}
\end{equation}
where all coefficients the same as in \eqref{cross-diffusion1}  and potential $V^{\eps}$ the same as in \eqref{regular A}. 
The distribution of particles $X_{k}^{i,\eps}(t)$ is represented by $\rho_k^{(1),N,\eps}(t)$, which is the first marginal of the joint law of $N$ particles in the $k$-th species. We will show that $\rho_k^{(1),N,\eps}$ converges to $\rho_k^\eps$ which is a measure-valued solution of \eqref{eps cross} when $N\to\infty$, and then $\rho_k^\eps$ converges to $\rho_k$ which is the weak solution of \eqref{cross-diffusion1} when $\eps\to0$.

Cross-diffusion systems have many applications in various fields, including biology, chemistry and population dynamics. We refer to \cite{J16} for more discussion on  cross-diffusion systems, especially with gradient flow structure. Assuming $\mathcal{A}$ is a functional of $\rho=(\rho_1,\ldots,\rho_n)$ given by
$$
\mathcal{A}[
\rho]=\int_{\R^d}e(\rho)\d x,
$$
the system can be represented as 
$$
\p_t \rho=\nabla\cdot \Big(B\nabla \frac{\delta \mathcal{A}}{\delta\rho}\Big),
$$
where the diffusion matrix $B(\rho) $ is positive and semidefinite. We have adopted this gradient flow structure and chosen a suitable functional $\mathcal{A}$ representing nonlinear diffusion modelling cell dynamics with volume exclusion \cite{CC06,CHS18,DB15} in tissue growth. In this work, we consider cross-diffusion models with more general pressure $P(\rho)=\big(\sum_{k=1}^{n}a_k\rho_k\big)^{m-1}$ where $m\geq2$ compared to \cite{LLP17,DJ20,DHJ23} with the addition of linear diffusion. These cross-diffusion systems are also related to aggregation-diffusion used in mathematical biology \cite{BGH85,BCPS20,CFSS18}. 

To derive macroscopic models from microscopic dynamics, one way is to take suitable scaling limit as the number of particles diverges. The mean field limit is one of the widely considered regimes. 
For deterministic cases, \cite{G16} provides a comprehensive review, while stochastic cases are discussed in \cite{S91,JW17}. In the stochastic case, the $N$-particle microscopic dynamics is governed by SDEs (1st order system) as
$$
\d X^i(t)=\frac{1}{N}\sum_{j=1}^NK\big(X^i(t)-X^j(t)\big)\d t+\sqrt{2\sigma}\d B^i(t),
$$
where $K$ is interaction kernel and $B^i$ are i.i.d standard Brownian motion. Oelschl\"{a}ger proposed the moderate interaction scaling as 
$$
\d X^i(t)=-\frac{1}{N}\sum_{j=1}^N\nabla W^N\big(X^i(t)-X^j(t)\big)\d t+\sqrt{2\sigma}\d B^i(t),
$$
where the kernel is written in gradient form and depends on the number of particles \cite{O85}. A usual choice is $W^N(x)=N^{\beta d}W(N^{\beta }x)$ where $\beta\in(0,\frac{1}{d+2})$, and $W^N$ converges to a Dirac delta $\delta_0$ when $N$ goes to infinity. 
The term "moderate" means that this nonlocal interaction is more local than the mean field regime, but when $N$ goes to infinity, particles will move in an asymptotically deterministic force field. Oelschl\"{a}ger \cite{O85,O90} rigorously derived the viscous porous medium equation from this moderate interaction. This interaction regime has recently seen rising interest, see for example \cite{JM98,BCM07,FLO19,GL23} and its application on cell-cell adhesion \cite{CHS18,CMSTT19,FBC23,DHPP24}.

In the realm of deriving cross-diffusion systems from interacting particle systems, the literature is growing. Notably, \cite{S00} focuses on the chemotaxis models and \cite{IRS12} deals with reaction-diffusion equations. It considers Maxwell–Stefan equation as the hydrodynamic limit of two-component Brownian particles in \cite{S18}, while \cite{FM15} derives  nonlocal Lotka-Volterra cross diffusion
system as large population limit of point measure-value Markov processes. \cite{DDD19} derives cross-diffusion systems of Shigesada–Kawasaki–Teramoto (SKT) type from Markov processes with mean-field scaling. Moreover, \cite{CDJ19} adopts the idea of moderate interaction, where they prove the many-particle system converges to an intermediate nonlocal diffusion system $(N\to\infty)$, and then obtain local cross-diffusion system when interaction potentials approach the Dirac delta distribution $(\eps\to0)$.
Further work \cite{CDHJ21}
derives SKT type cross-diffusion system from stochastic particle system, where a two-step limit is also applied. 
The cross-diffusion systems considered in \cite{CDJ19} are of the form
$$
\begin{cases}
\displaystyle\frac{ \p\rho_k}{\p t} =\nabla\cdot\Big(\sum_{l=1}^n a_{k l} \rho_k \nabla \rho_l\Big)+\sigma_k \Delta \rho_k\\
\rho_{k}(0)=\rho_{k,0}.
\end{cases}
$$
for smooth initial data, $\rho_{k,0}\in H^s(\R^d)$ with $s>d/2+1$, sufficiently small.

Our work generalises the paper by Figalli and Philipowski \cite{FP08}, where they consider the single-species viscous porous medium equation with exponent  $m>1$. They generalised the result of \cite{P07} and \cite{O01} for $m=2$ and proved that the (very) weak solution of the viscous porous medium equation can be obtained by the limit of solutions of the following nonlocal equations
$$
\frac{\p \rho^\eps}{\p t}=\nabla\cdot\big(\rho^\eps\nabla(V^\eps\ast(V^\eps\ast \rho^\eps)^{m-1})\big)+\Delta \rho^\eps.
$$
The connection between nonlocal equations and porous-medium type equations sheds light on the particle approximation (cf. \cite{Patacchini_blob19,BE23,CEW23,M20}). The authors in \cite{FP08} also derive the above nonlocal equation from stochastic particle system, thereby showing the convergence of the particle approximation of viscous porous medium equation.  We generalise the results to multi-species setting with $m\geq 2$, where the evolution of different species is coupled and has a more complicated structure. To prove the existence, we use nonlocal to local approximation, which can cover the cases with different mobility $b_k$. The higher regularity of each density is obtained by the linear diffusion, which is crucial to show the strong convergence in $L^1$.    While the proof of  uniqueness is more delicate where we have to assume the same mobility for different species. The result comes from the important observation that the sum of the densities satisfies a viscous porous medium equation, and the evolution of each species can be considered as a linear Fokker-Planck equation with the fixed pressure.  In terms of particle approximation, as in \cite{FP08}, we also obtain a logarithmic scale relation between the number of particles $N$ and the localisation parameter $\eps$, and present it under the framework of propagation of chaos.


The paper is organised as follows. Section \ref{main section}  is dedicated to introducing some notations and  presenting our main result; Section \ref{proof poc-cross} delves into the error estimate between moderately interacting particle system and associated nonlinear nonlocal process; in Section \ref{proof exist rhok} we investigate the convergence from nonlocal to local cross-diffusion system, which implies the existence of the limiting cross-diffusion system. Section \ref{proof of uniqueness} shows the uniqueness of the cross-diffusion system.

\section{Preliminaries and main result}\label{main section}
Let $\mathcal{M}(\R^d)$ be the space of probability measure equipped with the following metric which measures the weak convergence in $\mathcal{M}(\R^{d})$, for $\mu_1,\mu_2\in \mathcal{M}(\R^d)$
$$
d(\mu_1,\mu_2):=
\sup_{f\in BL}\Big|\int_{\R^d}f(x)\mu_1(\d x)-\int_{\R^d}f(x)\mu_2(\d x)\Big|,
$$
where the function space $BL$ denotes the set of functions which are bounded with Lipschitz constant 1.  $B_R$ denotes the closed ball in $\R^d $ centred at origin with radius $R$, while $B_R^c$ is its complementary set. 

Recall the definition of the stochastic particle systems for species $k=1,2,\ldots,n$ and $i=1,\ldots,N$ as in \eqref{cross-X},
$$
\begin{cases}
\d X_{k}^{i,\eps}(t)  =-b_k\Big[\nabla V^{\eps} \ast \Big(\displaystyle\sum_{l=1}^{n}\frac{a_{l}}{N}\sum_{j=1}^NV^{\eps} \big(\cdot-X_l^{j,\eps}(t)\big)\Big)^{m-1}\Big](X_{k}^{i,\eps}(t) ) \d t+\sqrt{2\sigma }\d B_k^i(t), \\
X_k^{i,\eps}(0)=\xi_k^i.
\end{cases}
$$
 The map $s\mapsto s^{m-1}$ is Lipschitz continuous   when 
 $m\geq 2$, and $V^\eps$ is bounded when $\eps$ is fixed. Then the existence and uniqueness of strong solution of \eqref{cross-X}  follow by standard SDE theory \cite[Theorem 3.1.1]{PR07}. 
 In addition, we introduce the associated McKean-Vlasov type nonlinear process $Y^{i,\eps}=(Y_1^{i,\eps},\ldots,Y_n^{i,\eps})$ satisfying the SDE below,  for any $k$,
\begin{equation}\label{cross-Y}
\begin{cases}\d Y_k^{i,\eps}(t)  =-b_k\Big[\nabla V^{\eps} \ast \big(\displaystyle\sum_{l=1}^{n}a_lV^{\eps} \ast  \rho_l^{\eps}\big)^{m-1}\Big](Y_k^{i,\eps}(t)) \d t+ \sqrt{2\sigma }\d B_k^{i}(t),  \\  Y_k^{i,\eps}(0)=\xi_k^i,\vspace{1.2ex}\\
\law(Y_k^{i,\eps}(t))=\rho^\eps_k(t),
\end{cases}
\end{equation}
where we choose the random variables $\xi_k^i$ and Brownian motion $B^{i}_k(t)$ the same as in \eqref{cross-X}.  
For every fixed $\eps$, $\sum_{l=1}^{n}a_lV^{\eps} \ast  \rho_l^{\eps}$ is a bounded finite measure and  $\nabla V^\eps$ is compactly supported, which implies the Lipschitz continuity holds as 
$$
\begin{aligned}
 &\, \Big|\nabla V^{\eps} \ast \Big(\sum_{l=1}^{n}a_lV^{\eps} \ast  \rho_l^{\eps}\Big)^{m-1}(x)-\nabla V^{\eps} \ast \Big(\sum_{l=1}^{n}a_lV^{\eps} \ast  \rho_l^{\eps}\Big)^{m-1}(y)\Big|\\
\leq  &\,  \int_{\R^d} \big|\nabla V^{\eps}(x-z)-\nabla V^{\eps}(y-z)\big|\Big(\sum_{l=1}^{n}a_lV^{\eps} \ast  \rho_l^{\eps}\Big)^{m-1}\d z\\
\leq  &\,  C_\eps|x-y|.
\end{aligned} $$
Then the existence and uniqueness hold for solutions of \eqref{cross-Y}, both trajectorially and in law \cite[Theorem 1.1]{S91}. 
We abused the notation a bit that we use $(\rho_1^\eps,\ldots,\rho_n^\eps)$ denoting the distribution of solution of SDE $(Y_1^{i,\eps},\ldots,Y_n^{i,\eps})$ in \eqref{cross-Y}. But we notice that, fixing $\eps$ for any $k$ and applying It\^{o}'s formula, the distribution $\rho_k^\eps$ coincides with the solution of the nonlocal equation \eqref{eps cross} formally as
$$
\frac{\p \rho_k^{\eps}}{\p t}=\nabla\cdot \Big(b_k\rho_k^{\eps}\nabla V^{\eps} \ast \big(\sum_{l=1}^{n}a_lV^{\eps} \ast  \rho_l^{\eps}\big)^{m-1}\Big)+\sigma  \Delta \rho_k^{\eps}.    
$$
The well-posedness of nonlinear processes \eqref{cross-Y} implies the following proposition.
\begin{proposition}
Assume initial data $\rho_{k,0}$ is a probability measure and with density $\rho_{k,0}\in L^1\cap L^\infty(\R^d)$, then there exists a measure-valued solution $\rho^\eps_k\in C([0,T],\mathcal{M}(\R^d))$ of \eqref{eps cross}.
\end{proposition}
Actually, we can obtain higher regularity of solutions of the nonlocal intermediate PDE \eqref{eps cross}, but the statement in the proposition above is enough for our argument in this paper.  

The quantitative error estimate between particles and nonlinear process is as follows.
\begin{proposition}\label{poc-cross}
Under the assumptions above, the distance between the strong solutions of SDEs \eqref{cross-X} and \eqref{cross-Y} can be estimated as, for fixed $\eps$
$$
\sum_{k=1}^n\E\left[\sup _{0 \leq s\leq t}\left|X_{k}^{i,\eps}(s)-Y_{k}^{i,\eps}(s)\right|^2\right] \leq \frac{C(\eps,t)}{N},
$$
where the constant $C(\eps,t)$ can be made explicitly.
\end{proposition}
See Section \ref{proof poc-cross} for the proof of this proposition. In terms of the distribution of particles, we have the following remark.
\begin{remark}
    
By the definition of 2-Wasserstein metric, for any $k$-th species, the distance between the one-particle distribution $\rho_{k}^{(1),N,\eps}(t)=\law(X_{k}^{i,\eps}(t))$ and $\rho_k^{\eps}(t)=\law(Y_{k}^{i,\eps}(t))$ can be estimated as follows,  for $t\in[0,T]$,
$$
\begin{aligned}
 &\, W_2^2(\rho_{k}^{(1),N,\eps}(t),\rho_{k}^{\eps}(t))\leq \sum_{k=1}^n W_2^2(\rho_{k}^{(1),N,\eps}(t),\rho_{k}^{\eps}(t))\leq \sum_{k=1}^n\E\left[\left|X_{k}^{i,\eps}(t)-Y_{k}^{i,\eps}(t)\right|^2\right] \\ \leq  &\, \sum_{k=1}^n\E\left[\sup _{0 \leq s\leq T}\left|X_{k}^{i,\eps}(s)-Y_{k}^{i,\eps}(s)\right|^2\right] \leq \frac{C(\eps,T)}{N}.
\end{aligned}
$$
\end{remark}
\begin{remark}\label{Neps}
According to Proposition \ref{poc-cross} and the expression of $C(\eps,t)$ (see \eqref{eps N relation}), we can take suitable logarithmic dependence of $\eps$ and $N$ as $\eps=\eps(N)$ which goes to $0$ when $N$ goes to $\infty$. Then it holds for $t\in[0,T]$,
$$
W_2(\rho_{k}^{(1),N,\eps(N)}(t),\rho_{k}^{\eps(N)}(t))\to 0, \text{ as } N\to\infty.
$$
\end{remark}
Now we define the weak solution of the cross-diffusion system \eqref{cross-diffusion1}:
\begin{definition}\label{definition}
A weak solution $\rho=(\rho_1,\ldots,\rho_n)$ of the cross-diffusion system \eqref{cross-diffusion1} on the time interval $[0,T]$ satisfies that, for each species $k$,
    \begin{itemize}
    \item [(1)] $\rho_k\in C([0,T],\mathcal{M}(\R^d))$ is a measure-valued solution with initial data $\rho_{k,0}\in L^1\cap L^\infty(\R^d)$;
        \item[(2)] for almost every $t\in[0,T]$, $\rho_k(t)$ is absolutely continuous with respect to Lebesgue measure (for simplicity which is also denoted by $\rho_k(t)$),  and  $\rho_k\in L^m([0,T]\times \R^d)$; 
        \item[(3)] $\big(\sum_{l=1}^{n}a_{l} \rho_l\big)^{m-1}\in L^{\frac{m}{m-1}}(0,T;W^{1,\frac{m}{m-1}}(\R^d))$;
     \item[(4)] for almost any $t\in[0,T]$ and $f\in C^1([0,T], C^2_b(\R^d))$, it holds
     \begin{equation}\label{weaksolution}
\begin{aligned}
      &\int_{\R^d}f(t,x)\rho_k(t,x)
\d x+  \int_0^t\int_{\R^d}\p_s f(s,x)\rho_k(s,x)\d x\d s=\int_{\R^d} f(0,x)\rho_{k,0}(x)\d x   \\
  +&\sigma \int_0^t\int_{\R^d}\Delta f(s,x) \rho_k(s,x)\d x\d s  -\int_0^t\int_{\R^d}b_k\rho_k(s,x)\nabla f(s,x) \cdot\nabla\big(\sum_{l=1}^{n}a_{l} \rho_l(s,x)\big)^{m-1} \d x\d s.
    \end{aligned}     
     \end{equation}
    \end{itemize}
    \end{definition}
The following theorems give the well-posedness of the cross-diffusion system \eqref{cross-diffusion1}.
\begin{theorem}[Existence]\label{exist rhok}
Up to a subsequence, the solutions of nonlocal equation \eqref{eps cross} $(\rho^{\eps})_{\eps>0}$ converges in $C([0,T],\mathcal{M}(\R^d))$ to $\rho$ of a weak solution of \eqref{cross-diffusion1}.    
\end{theorem}
See Section \ref{proof exist rhok} for the proof of this theorem. 
\begin{remark}
We emphasise that Theorem \ref{exist rhok} characterise all possible adherence points of the convergent subsequences as $\eps$ goes to 0, as weak solutions of the cross-diffusion system \eqref{cross-diffusion1}.
\end{remark}
\begin{theorem}[Uniqueness]\label{uniqueness}
 If we assume further that all the species have the same mobility, i.e., $b_1=\ldots=b_n=b>0$, then there exists a unique weak solution of the cross-diffusion equation \eqref{cross-diffusion1} defined as in Definition \ref{definition}.
\end{theorem}

We give the proof of this result in Section \ref{proof of uniqueness}. As the direct consequence of Proposition \ref{poc-cross} and Theorem \ref{exist rhok}, Theorem \ref{limit result} together with Corollary \ref{poc} is our second main result.  
\begin{theorem}[Particle Approximation]\label{limit result}
Under the assumptions of Theorem \ref{uniqueness}, for almost any $t\in[0,T]$ and any species $k$, the distribution of the particle \eqref{cross-X} converges to the weak solution of the cross-diffusion system \eqref{cross-diffusion1} when $N$ goes to infinity, and then  $\eps$ goes to 0 that
    $$
\lim_{\eps\to0}\lim_{N\to\infty}\rho_k^{(1),N,\eps}(t)=\rho_k(t).
    $$
\end{theorem}
In fact, we can take $\eps$ depending on $N$ as in Remark \ref{Neps} and combine the two-step limit into one as 
$
\lim_{N\to\infty}\rho_k^{(1),N,\eps(N)}(t)=\rho_k(t).
$

Let $M$ be a fixed natural number and $\rho_k^{(M),N,\eps}$ is the joint law of $X^{i,\eps}_k$, $i=1,2,\ldots,M$ on $\R^{dM}$, i.e. the $M$-marginal of the joint law of $N$ particles. And $\rho_k^{\otimes M}$ is the independently tensorised solution of cross-diffusion system on $\R^{dM}$, we obtain propagation of chaos.
 \begin{corollary}\label{poc}
  Under the assumptions of Theorem \ref{limit result}, it holds
  $$
    \lim_{\eps\to0}\lim_{N\to\infty}\rho_k^{(M),N,\eps}(t)=\rho_k^{\otimes M}(t).
    $$
 \end{corollary}

\section{Proof of Proposition \ref{poc-cross}}\label{proof poc-cross}
In this section, we investigate the large $N$ limit of the particle system \eqref{cross-X}. In particularly, we will prove the convergence $\lim_{N\to\infty}\rho^{(1),N,\eps}_k=\rho^\eps_k$.

Since $X_k^{i,\eps}(0)=Y_k^{i,\eps}(0)=\xi_k^i$, H\"{o}lder's inequality implies
$$
\begin{aligned}
\big|X_{k}^{i,\eps}(t)-Y_{k}^{i,\eps}(t)\big|^2
= &\, \,\Big|\int_0^t\int_{\R^d}b_k\nabla V^\eps(z)\Big[\Big(\sum_{l=1}^{n}\frac{a_{l}}{N}\sum_{j=1}^NV^{\eps} \big(X_{k}^{i,\eps}(s)-X_l^{j,\eps}(s)-z\big)\Big)^{m-1}\\ &\, 
\qquad\qquad\qquad\qquad\qquad-\Big(\sum_{l=1}^{n}a_{l}V^{\eps} \ast  \rho_l^{\eps}\big(s,Y_{k}^{i,\eps}(s)-z\big)\Big)^{m-1}\Big]\d z\d s\Big|^2\\
\leq  &\,  t \int_0^t\Big(\int_{\R^d}b_k|\nabla V^\eps(z)|
\Big[\Big(\sum_{l=1}^{n}\frac{a_{l}}{N}\sum_{j=1}^NV^{\eps} \big(X_{k}^{i,\eps}(s)-X_l^{j,\eps}(s)-z\big)\Big)^{m-1}\\ &\, 
\qquad\qquad\qquad\qquad\qquad-\Big(\sum_{l=1}^{n}a_{l}V^{\eps} \ast  \rho_l^{\eps}\big(s,Y_{k}^{i,\eps}(s)-z\big)\Big)^{m-1}\Big]\d z\Big)^2\d s.
\end{aligned}
$$
The following equality holds  by the scaling of $V^\eps(\cdot)=\eps^{-d}V(\cdot/\eps)$,
$$
\int_{\R^d}|\nabla V^\eps(z) |\d z=\frac{1}{\eps}\int_{\R^d}|\nabla V(z) |\d z<\frac{C_V}{\eps},
$$
where $C_V$ is independent with $\eps$, then
\begin{equation}\label{XYdistance}
\begin{aligned}
&\,\big|X_{k}^{i,\eps}(t)-Y_{k}^{i,\eps}(t)\big|^2\\
\leq  &\,  t b_k^2\int_0^t\Big|\int_{\R^d}\nabla V^\eps(z)\d z\times
\sup_{y\in \R^d}\Big[\Big(\sum_{l=1}^{n}\frac{a_{l}}{N}\sum_{j=1}^NV^{\eps} \big(X_{k}^{i,\eps}(s)-X_l^{j,\eps}(s)-y\big)\Big)^{m-1}\\ &\, 
\qquad\qquad\qquad\qquad\qquad\qquad\qquad\qquad-\Big(\sum_{l=1}^{n}a_{l}V^{\eps} \ast  \rho_l^{\eps}\big(s,Y_{k}^{i,\eps}(s)-y\big)\Big)^{m-1}\Big]\Big|^2\d s\\
\leq  &\,  \frac{C_V^2b_k^2t}{\eps^2} \int_0^t\sup_{y\in \R^d}\Big|\Big(\sum_{l=1}^{n}\frac{a_{l}}{N}\sum_{j=1}^NV^{\eps} \big(X_{k}^{i,\eps}(s)-X_l^{j,\eps}(s)-y\big)\Big)^{m-1}\\ &\, 
\qquad\qquad\qquad\qquad\qquad\qquad\qquad\qquad-\Big(\sum_{l=1}^{n}a_{l}V^{\eps} \ast  \rho_l^{\eps}\big(s,Y_{k}^{i,\eps}(s)-y\big)\Big)^{m-1}\Big|^2\d s
\end{aligned}
\end{equation}
When $m\geq 2$, it holds that
$$
\begin{aligned}
&\,\Big|\Big(\sum_{l=1}^{n}\frac{a_{l}}{N}\sum_{j=1}^NV^{\eps} \big(X_{k}^{i,\eps}(s)-X_l^{j,\eps}(s)-y\big)\Big)^{m-1}-\Big(\sum_{l=1}^{n}a_{l}V^{\eps} \ast  \rho_l^{\eps}\big(s,Y_{k}^{i,\eps}(s)-y\big)\Big)^{m-1}\Big|^2   \\
\leq &\, C\|V^\eps\|^{2m-4}_{L^\infty}\Big|\sum_{l=1}^{n}\frac{a_l}{N}\sum_{j=1}^NV^{\eps} \big(X_{k}^{i,\eps}(s)-X_l^{j,\eps}(s)-y\big)-\sum_{l=1}^{n}a_{l}V^{\eps} \ast  \rho_l^{\eps}\big(s,Y_{k}^{i,\eps}(s)-y\big)\Big|^2.
\end{aligned}
$$
And the quadratic term can be estimated as follows 
$$
\begin{aligned}
 &\, \Big|\sum_{l=1}^{n}\frac{a_l}{N}\sum_{j=1}^NV^{\eps} \big(X_{k}^{i,\eps}(s)-X_l^{j,\eps}(s)-y\big)-\sum_{l=1}^{n}a_{l}V^{\eps} \ast  \rho_l^{\eps}\big(s,Y_{k}^{i,\eps}(s)-y\big)\Big|^2\\
\leq  &\,  n\sum_{l=1}^{n}a_{l}^2\Big|\frac{1}{N}\sum_{j=1}^NV^{\eps} \big(X_{k}^{i,\eps}(s)-X_l^{j,\eps}(s)-y\big)-V^{\eps} \ast  \rho_l^{\eps}\big(s,Y_{k}^{i,\eps}(s)-y\big)\Big|^2\\
\leq  &\,  3n\sum_{l=1}^{n}a^2_{l}\Big|\frac{1}{N}\sum_{j=1}^NV^{\eps} \big(X_{k}^{i,\eps}(s)-X_l^{j,\eps}(s)-y\big)-\frac{1}{N}\sum_{j=1}^NV^{\eps} \big(X_{k}^{i,\eps}(s)-Y_l^{j,\eps}(s)-y\big)\Big|^2 \\
  &\,  +3n\sum_{l=1}^{n}a_{l}^2\Big|\frac{1}{N}\sum_{j=1}^NV^{\eps} \big(X_{k}^{i,\eps}(s)-Y_l^{j,\eps}(s)-y\big)-\frac{1}{N}\sum_{j=1}^NV^{\eps} \big(Y_{k}^{i,\eps}(s)-Y_l^{j,\eps}(s)-y\big)\Big|^2 \\
   &\,  +3n\sum_{l=1}^{n}a_{l}^2\Big|\frac{1}{N}\sum_{j=1}^NV^{\eps} \big(Y_{k}^{i,\eps}(s)-Y_l^{j,\eps}(s)-y\big)-V^{\eps} \ast  \rho_l^{\eps}\big(s,Y_{k}^{i,\eps}(s)-y\big)\Big|^2 \\
  =: &\, J^{i,k}_1+J^{i,k}_2+J^{i,k}_3.
\end{aligned}
$$
The first term and the second term can be estimated thanks to the Lipschitz continuity of $V^\eps$ fixed $\eps$ as 
$$
J^{i,k}_1\leq 3n\|\nabla V^\eps\|^2_{L^\infty}\sum_{l=1}^{n} \frac{a_{l}^2}{N}\sum_{j=1}^N\big|X_l^{j,\eps}(s)-Y_l^{j,\eps}(s)\big|^2, $$
and
$$J^{i,k}_2\leq 3n\|\nabla V^\eps\|^2_{L^\infty} \big|X_k^{i,\eps}(s)-Y_k^{i,\eps}(s)\big|^2\sum_{l=1}^{n}a_{l}^2.
$$
For any $t\in[0,T]$, we take expectation value of \eqref{XYdistance} that
we then have
$$
\begin{aligned}
 &\, \E\Big[\sup_{0\leq s\leq t} |X_{k}^{i,\eps}(s)-Y_{k}^{i,\eps}(s)|^2\Big]\\
\leq  &\,   C_0(\eps)t
\int_0^t\sup_{y\in \R^d}\E\Big[\Big|\sum_{l=1}^{n}\frac{a_l}{N}\sum_{j=1}^NV^{\eps} \big(X_{k}^{i,\eps}(s)-X_l^{j,\eps}(s)-y\big)-\sum_{l=1}^{n}a_{l}V^{\eps} \ast  \rho_l^{\eps}\big(s,Y_{k}^{i,\eps}(s)-y\big)\Big|^2\Big]\d s\\
\leq  &\,   C_0(\eps)t
\int_0^t\sup_{y\in \R^d}\big(\E[J^{i,k}_1]+\E[J^{i,k}_2]+\E[J^{i,k}_3]\big)\d s,
\end{aligned}
$$
where 
$$
\begin{aligned}
\E [J^{i,k}_3]= &\, 3n\sum_{l=1}^{n}a_{l}^2\E\Big[\Big|\frac{1}{N}\sum_{j=1}^NV^{\eps} \big(Y_{k}^{i,\eps}(s)-Y_l^{j,\eps}(s)-y\big)-V^{\eps} \ast  \rho_l^{\eps}\big(s,Y_{k}^{i,\eps}(s)-y\big)\Big|^2\Big]\\
\leq  &\,  3n\sum_{l=1}^na_{l}^2\frac{1}{N^2}\sum_{j,j'}\E\Big[\Big(V^{\eps} \big(Y_{k}^{i,\eps}(s)-Y_l^{j,\eps}(s)-y\big)-V^{\eps} \ast  \rho_l^{\eps}\big(s,Y_{k}^{i,\eps}(s)-y\big)\Big)\\ &\, \qquad\qquad\qquad\times\Big(V^{\eps} \big(Y_{k}^{i,\eps}(s)-Y_l^{j',\eps}(s)-y\big)-V^{\eps} \ast  \rho_l^{\eps}\big(s,Y_{k}^{i,\eps}(s)-y\big)\Big)\Big],
\end{aligned}
$$
and the constant 
$$
C_0(\eps)\sim \frac{\|V^\eps\|_{L^\infty}^{2m-4}}{\eps^2}\sim \frac{1}{\eps^{2md-4d+2}}.
$$
Recalling the definition of nonlinear process  \eqref{cross-Y}, we can see that the randomness of $Y_{l}^{j,\eps}$ for different index $j$ comes from i.i.d Brownian motions $B^j_l$.  When $i\neq j\neq j'$, the sum vanishes because 
$$
\begin{aligned}
 &\, \E\Big[\Big(V^{\eps} \big(Y_{k}^{i,\eps}(s)-Y_l^{j,\eps}(s)-y\big)-V^{\eps} \ast  \rho_l^{\eps}\big(s,Y_{k}^{i,\eps}(s)-y\big)\Big)\\ &\, \qquad\qquad\times\Big(V^{\eps} \big(Y_{k}^{i,\eps}(s)-Y_l^{j',\eps}(s)-y\big)-V^{\eps} \ast  \rho_l^{\eps}\big(s,Y_{k}^{i,\eps}(s)-y\big)\Big)\Big]\\
=  &\,  \E\bigg[\E\Big[\Big(V^{\eps} \big(z-Y_l^{j,\eps}(s)-y\big)-V^{\eps} \ast  \rho_l^{\eps}\big(s,z-y\big)\Big)\\ &\, \qquad\qquad\times\Big(V^{\eps} \big(z-Y_l^{j',\eps}(s)-y\big)-V^{\eps} \ast  \rho_l^{\eps}\big(s,z-y\big)\Big)\big|z=Y_{k}^{i,\eps}(s)\Big]\bigg]\\
=  &\,  \E\bigg[\E\Big[V^{\eps} \big(z-Y_l^{j,\eps}(s)-y\big)-V^{\eps} \ast  \rho_l^{\eps}\big(s,z-y\big)\big|z=Y_{k}^{i,\eps}\Big]\\ &\, \qquad\qquad\times\E\Big[V^{\eps} \big(z-Y_l^{j',\eps}(s)-y\big)-V^{\eps} \ast  \rho_l^{\eps}\big(s,z-y\big)\big|z=Y_{k}^{i,\eps}\Big]\bigg]
=  0,
\end{aligned}
$$ 
where $Y_l^{j,\eps}(s)$ and $Y_l^{j',\eps}(s)$ have the same  distribution $\rho_l^{\eps}(s)$. 
Fixed the index $i$, number of elements in the set $$S=\{j,j'|\text{At least two of indexes } i,j,j' \text{are equal}\}
$$
is $3N-2$.  
Thus we can bound $\E[J^{i,k}_3]$ as
$$
\begin{aligned}
\E [J^{i,k}_3] 
=  &\,  3n\sum_{l=1}^na_{l}^2\frac{1}{N^2}\sum_{i\neq j\neq j'}\E\Big[\Big(V^{\eps} \big(Y_{k}^{i,\eps}(s)-Y_l^{j,\eps}(s)-y\big)-V^{\eps} \ast  \rho_l^{\eps}\big(s,Y_{k}^{i,\eps}(s)-y\big)\Big)\\ &\, \qquad\qquad\qquad\times\Big(V^{\eps} \big(Y_{k}^{i,\eps}(s)-Y_l^{j',\eps}(s)-y\big)-V^{\eps} \ast  \rho_l^{\eps}\big(s,Y_{k}^{i,\eps}(s)-y\big)\Big)\Big]\\
&\,  +3n\sum_{l=1}^na_{l}^2\frac{1}{N^2}\sum_{S}\E\Big[\Big(V^{\eps} \big(Y_{k}^{i,\eps}(s)-Y_l^{j,\eps}(s)-y\big)-V^{\eps} \ast  \rho_l^{\eps}\big(s,Y_{k}^{i,\eps}(s)-y\big)\Big)\\ &\, \qquad\qquad\qquad\times\Big(V^{\eps} \big(Y_{k}^{i,\eps}(s)-Y_l^{j',\eps}(s)-y\big)-V^{\eps} \ast  \rho_l^{\eps}\big(s,Y_{k}^{i,\eps}(s)-y\big)\Big)\Big]\\
\leq &\, \frac{12(3N-2)n^2\|V^\eps\|^2_{L^\infty}\sum_{l=1}^{n}a_{l}^2}{N^2}.
\end{aligned}
$$
Now we possess all ingredients to estimate $\E\Big[\sup_{0\leq s\leq t} |X_{k}^{i,\eps}(s)-Y_{k}^{i,\eps}(s)|^2\Big]$ as
$$
\begin{aligned}
 &\, \E\Big[\sup_{0\leq s\leq t} |X_{k}^{i,\eps}(s)-Y_{k}^{i,\eps}(s)|^2\Big]\\
\leq  &\,   C_0(\eps)t\int_0^t\sup_{y\in \R^d}\big(\E[J^{i,k}_1]+\E[J^{i,k}_2]+\E[J^{i,k}_3]\big)\d s\\
\leq  &\,  C_0(\eps)t\int_0^t\Big(3n\|\nabla V^\eps\|^2_{L^\infty}\sum_{l=1}^{n}a_{l}^2 \E\big|X_l^{i,\eps}(s)-Y_l^{i,\eps}(s)\big|^2\\ &\,+3n\|\nabla V^\eps\|^2_{L^\infty} \E\big|X_k^{i,\eps}(s)-Y_k^{i,\eps}(s)\big|^2\sum_{l=1}^{n}a_{l}^2 +\frac{12(3N-2)n^2\|V^\eps\|^2_{L^\infty}\sum_{l=1}^{n}a_{l}^2}{N^2}\Big)\d s.
\end{aligned}
$$
We sum up species index $k$ from $1$ to $n$, then we can see that
$$
\begin{aligned}
 &\, \sum_{k=1}^n \E\Big[\sup_{0\leq s\leq t} |X_{k}^{i,\eps}(s)-Y_{k}^{i,\eps}(s)|^2\Big]
\leq \frac{C_1(\eps)t^2}{N} +C_2(\eps,T)\int_0^t\sum_{k=1}^n \E\big|X_k^{i,\eps}(s)-Y_k^{i,\eps}(s)\big|^2\d s, 
\end{aligned}
$$
where
$$
C_1(\eps)\sim C_0(\eps)\|V^\eps\|^2_{L^\infty}\sim\frac{1}{\eps^{2md-2d+2}},\quad C_2(\eps,T)\sim C_0(\eps)\|\nabla V^\eps\|^2_{L^\infty}\sim\frac{1}{\eps^{2md-2d+4}}.
$$
Gronwall's inequality implies the estimate as 
$$
\begin{aligned}
\sum_{k=1}^n \E\Big[\sup_{0\leq s\leq t} |X_{k}^{i,\eps}(s)-Y_{k}^{i,\eps}(s)|^2\Big]&\,\leq \frac{2C_1(\eps)}{N}\int_0^ts e^{-C_2(\eps,T)s}\d s\leq \frac{C(\eps,t)}{N},
\end{aligned}
$$
where 
\begin{equation}\label{eps N relation}
C(\eps,t)=\frac{2C_1(\eps)}{\big(C_2(\eps,T)\big)^2}e^{C_2(\eps,T)t}\sim\eps^{6+2d(m-1)}\exp(t/\eps^{4+2d(m-1)}).
\end{equation} 

\section{Proof of Theorem \ref{exist rhok}}\label{proof exist rhok}
In this section, we will prove the nonlocal to local convergence, i.e., for any species $k$, the measure-valued solution $\rho^\eps_k$ of equations \eqref{eps cross} 
converges to a weak solution $\rho_k$ of the cross-diffusion system \eqref{cross-diffusion1}
when $\eps$ goes to $0$ (up to a subsequence).

Let us define the nonnegative functions $g^\eps:[0,T]\times \R^d\to\R$ as
$$
g^\eps(t,x)=V^{\eps} \ast \big(\sum_{l=1}^{n}a_{l}V^{\eps} \ast  \rho_l^{\eps}(t)\big)^{m-1}(x),
$$
further define the regularised solution of nonlocal equation \eqref{eps cross} as $$u_{k}^{\eps}=V^\eps\ast \rho_k^{\eps},$$ which is also a nonnegative probability measure, then 
$
g^\eps=V^{\eps} \ast \big(\sum_{l=1}^{n}a_{l}u_{l}^{\eps}\big)^{m-1}.
$
Then we convolve both sides of \eqref{eps cross} with $V^\eps$ to obtain the equality
$$
\p_t u_{k}^{\eps}=\nabla\cdot\big(b_k\rho_k^{\eps}\nabla g^\eps\big)\ast  V^\eps+\sigma  \Delta u_{k}^{\eps}=\big(b_k\rho_k^{\eps}\nabla g^\eps\big)\ast \nabla V^\eps+\sigma \Delta u_{k}^{\eps},
$$
which leads to
$$
\p_t \big(a_ku_{k}^{\eps}\big)=\Big(a_kb_k\rho_k^{\eps}\nabla g^\eps\Big)\ast \nabla V^\eps+a_k\sigma\Delta u_{k}^{\eps}.
$$
Summing up species index $k$ from $1$ to $n$ and testing against it by $( \sum_ka_ku_{k}^{\eps})^{m-1}$, we get 
$$
\begin{aligned}
   &\, \int_{\R^d} (\sum_ka_ku_{k}^{\eps}(t))^{m}\d x-\int_{\R^d}(\sum_ka_ku_{k}^{\eps}(0))^{m}\d x \\
  = &\, \int_0^t\int_{\R^d}(\sum_ka_ku_{k}^{\eps})^{m-1}\Big(\sum_ka_kb_k\rho_k^{\eps}\nabla g^\eps\Big)\ast \nabla V^\eps\d x\d s \\&\,\qquad\qquad\qquad\qquad+\sigma\int_0^t\int_{\R^d}(\sum_ka_ku_{k}^{\eps})^{m-1}\Delta(\sum_k a_ku_{k}^{\eps})\d x\d s \\
  = &\, -\int_0^t\int_{\R^d}\nabla V^\eps\ast(\sum_ka_ku_{k}^{\eps})^{m-1}\cdot\Big(\sum_ka_kb_k\rho_k^{\eps}\nabla g^\eps\Big)\d x\d s \\ &\, \qquad\qquad\qquad\qquad-(m-1)\sigma\int_0^t\int_{\R^d}(\sum_ka_ku_{k}^{\eps})^{m-2}\big|\nabla(\sum_k a_ku_{k}^{\eps})\big|^2\d x\d s \\
  = &\, -\sum_k\int_0^t\int_{\R^d}a_kb_k\big|\nabla g^\eps\big|^2\rho_k^{\eps}(\d x)\d s -(m-1)\sigma\int_0^t\int_{\R^d}(\sum_ka_ku_{k}^{\eps})^{m-2}\big|\nabla(\sum_k a_ku_{k}^{\eps})\big|^2\d x\d s ,
\end{aligned}
$$
where we applied the following fact in the second equality, for some integrable $f$ and $h$,
$$
\int_{\R^d} f(x)\big(h \ast  \nabla V^{\eps}\big)(x) \d x=-\int_{\R^d}\big(\nabla V^{\eps} \ast  f\big)(x) h(x) \d x.
$$
By the assumption $\rho_{k,0}\in L^1\cap L^\infty\subset L^m$ and $u_{k}^{\eps}(0)=\rho_{k,0}\ast V^\eps$,  which implies 
$$
\|u_{k}^{\eps}(0)\|_{L^m(\R^d)} \leq\|\rho_{k,0}\|_{L^m(\R^d)}<\infty.
$$
Then we get the uniform in $\eps$ estimate as follows.

\begin{lemma}\label{apriori}
For each species $k$ and $t\geq 0$, the following estimate holds
$$
\begin{aligned}
 &\, \big\|\sum_ka_ku_{k}^{\eps}(t)\big\|_{L^m}^m+\sum_k\int_0^t\int_{\R^d}a_kb_k\big|\nabla g^\eps\big|^2\rho_k^{\eps}(s,\d x)\\ &\, \qquad\qquad\qquad+(m-1)\sigma\int_0^t\int_{\R^d}(\sum_ka_ku_{k}^{\eps})^{m-2}\big|\nabla(\sum_k a_ku_{k}^{\eps})\big|^2\leq\big\|\sum_ka_k\rho_{k,0}\big\|_{L^m}^m<\infty. 
\end{aligned}
$$
\end{lemma}

\begin{remark}\label{remark}
 From above lemma, we can deduce that for each $k$ the nonnegative sequence $(u_{k}^{\eps})_{\eps>0}$  is bounded in $L^{\infty}([0,T], L^m(\R^d))$.   And the equality
 $$
 \int_0^t\int_{\R^d} ( \sum_ka_ku_{k}^{\eps})^{m-2}\big|\nabla(\sum_ka_ku_{k}^{\eps})\big|^2 \d x\d s=\frac{4}{m^2}\int_0^t\int_{\R^d}\big|\nabla(\sum_ka_ku_{k}^{\eps})^{\frac{m}{2}}\big|^2 \d x\d s,
 $$
implies the sequence $\big((\sum_ka_ku_{k}^{\eps})^{m/2}\big)_{\eps>0}$ 
is bounded in $L^{2}([0,T], H^1(\R^d))$; and for each $k$,
$$
\int_0^t\int_{\R^d}\big|\nabla g^\eps(s,x)\big|^2\rho_k^{\eps}(s,\d x)\d s\text{ is uniformly bounded in }\eps. 
$$
Notice that we are not able to get higher regularity for $u_k^\eps$ from the estimate above, but only for the sum $\sum_ka_ku_{k}^{\eps}$. 
\end{remark}
We now state the following lemma.
\begin{lemma}\label{relative compact}
For each $k$, the sequence  $\left(\rho_k^{\eps}\right)_{\eps>0}$ is relatively compact in $C\big([0, T], \mathcal{M}(\R^d)\big)$.
\end{lemma}
\begin{proof}
To apply the Ascoli-Arzel\`{a} theorem, we need to verify the following two claims, for each  $k$-th species,
\begin{itemize}
    \item[(1)] there is a relatively compact subset $\mathcal{K}_k\subset\mathcal{M}(\R^d)$ which is independent of $\eps$ and $t$,  that for any $t\in[0,T]$ and $\eps>0$, $\rho_k^\eps(t)\in \mathcal{K}_k$ ;
    \item[(2)]  the sequence $(\rho_k^{\eps})_{\eps>0}$ is equicontinuous, i.e., for every $\eta>0$ there exists $\delta$ such that, for all  $\eps>0$ and $t,s\in[0,T]$ such that $ |t-s|<\delta$, then it implies
    $
d(\rho_k^{\eps}(s),\rho_k^{\eps}(t))<\eta.
    $
\end{itemize}

We start with proving the first statement. A subset of $\mathcal{M}(\R^d)$ is relatively compact if and only if it is tight, then it is equivalent to show  for any $t\in[0,T]$ and $\eta>0$, there exists a compact set $K_k\subset \R^d$ with $\rho_k^{\eps}(K_k)\geq 1-\eta$ for all $\eps>0$.  
Recall the nonlinear process $Y_k^{\eps}(t)$ defined by \eqref{cross-Y} with $\law(Y_k^{\eps}(t))=\rho_k^{\eps}(t)$ satisfies the SDE
$$
\d Y_k^{\eps}(t)=-b_k\nabla g^\eps(t,Y_k^{\eps}(t))\d t+\sqrt{2\sigma }\d B_k(t).
$$
Then $\rho_k^{\eps}(K_k)\geq 1-\eta$ is equivalent to $\mathbb{P}\big[Y_k^{\eps}(t)\in K_k^c\big]\leq \eta$. We can take the compact set as a closed ball with radius $R>0$, then the probability of $Y_k^{\eps}$ being outside the closed ball can be estimated as
$$
\begin{aligned}
\mathbb{P}\Big[|Y_k^{\eps}(t)|>R\Big]  &\,  = \mathbb{P}\left[\left|Y_k^{\eps}(0)-b_k\int_0^t \nabla g^{\eps}\left(s,Y_k^{\eps}(s)\right) \d s+\sqrt{2\sigma }B_k(t)\right|>R\right]\\
\leq  &\, \mathbb{P}\left[\left|Y_k^{\eps}(0)\right|>\frac{R}{3}\right]+\mathbb{P}\left[\left|b_k\int_0^t \nabla g^{\eps}\left(s,Y_k^{\eps}(s)\right) \d s\right|>\frac{R}{3}\right]+\mathbb{P}\left[\left|\sqrt{2\sigma }B_k(t)\right|>\frac{R}{3}\right],
\end{aligned}
$$
where the first term and the third term goes to $0$ as $R\to \infty$. For the second term, we deduce 
$$
\begin{aligned}
\mathbb{P}\left[\left|b_k\int_0^t \nabla g^{\eps}\left(s,Y_k^{\eps}(s)\right) \d s\right|>\frac{R}{3}\right]  &\,  \leq \frac{9}{R^2} \E\left[\left|b_k\int_0^t \nabla g^{\eps}\left( s,Y_k^{\eps}(s)\right) \d s\right|^2\right] \\
 &\,  \leq \frac{9 tb_k^2}{R^2} \E\left[\int_0^t\left|\nabla g^{\eps}\left(s,Y_k^{\eps}(s)\right)\right|^2 d s\right] \\
 &\,  =\frac{9 tb_k^2}{R^2} \int_0^t \int_{\R^d}\big|\nabla g^{\eps}(s, x)\big|^2 \rho_k^{\eps}(s, \d x) \d s,
\end{aligned}
$$
which goes to 0 by sending $R$ to $\infty$ by Remark  \ref{remark}.

Now we prove the second claim. For $s,t\in [0,T]$, the distance between $\rho_k^{\eps}(s)$ and $\rho_k^{\eps}(t)$ has the following estimate
$$
\begin{aligned}
d\left(\rho_k^{\eps}(t), \rho_k^{\eps}(s)\right)  &\,  =\sup_{f \in B L}\left|\int_{\R^d} f(x) \rho_k^{\eps}(t,\d x)-\int_{\R^d} f(x) \rho_k^{\eps}(s, \d x)\right| \\
 &\,  =\sup _{f \in B L}\big|\E\left[f\big(Y_k^{\eps}(t)\big)\right]-\E\left[f\big(Y_k^{\eps}(s)\big)\right]\big| \\
 &\,  \leq \Big(\E\big[\big|Y_k^{\eps}(t)-Y_k^{\eps}(s)\big|^2\big]\Big)^{1 / 2}, \\
 \end{aligned}
$$
and by Minkowski's inequality 
$$
\begin{aligned}
\Big(\E\big[\big|Y_k^{\eps}(t)-Y_k^{\eps}(s)\big|^2\big]&\,\Big)^{1 / 2}
=  \Big(\E\Big[\Big|b_k\int_s^t \nabla g^{\eps}\left(r, Y_k^{\eps}(r)\right) \d r+\sqrt{2\sigma }B_k(t)-\sqrt{2\sigma }B_k(s)\Big|^2\Big]\Big)^{1 / 2} \\
  \leq &\, b_k\Big(\E\Big[\Big|\int_s^t \nabla g^{\eps}\big(r,Y_k^{\eps}(r)\big) \d r\Big|^2\Big]\Big)^{1 / 2}+\sqrt{2\sigma }\Big(\E\left[\left|B_k(t)-B_k(s)\right|^2\right]\Big)^{1 / 2} \\
   \leq &\,b_k\Big(\E\Big[|t-s| \int_s^t\left|\nabla g^{\eps}\big(r,Y_k^{\eps}(r)\big)\right|^2 \d r\Big]\Big)^{1 / 2}+\sqrt{2\sigma }|t-s|^{1 / 2} \\
  =&\, |t-s|^{1 / 2}\Big[b_k\Big[\int_s^t \int_{\R^d}\big|\nabla g^{\eps}(r,x )\big|^2 \rho_k^{\eps}(r, \d x) \d r\Big)^{1 / 2}+\sqrt{2\sigma }\Big] \\
  \leq &\, C|t-s|^{1 / 2} ,    
\end{aligned}
$$
where the constant $C$ is independent with $\eps$ by Remark  \ref{remark} again. In conclusion, Lemma \ref{relative compact} is proved. 
\end{proof}
We have shown that for each species $k$ the sequence $\left(\rho_k^{\eps}\right)_{\eps>0}$ has a convergent subsequence. We now fix such a convergent subsequence, which is still denoted by $\left(\rho_k^{\eps}\right)_{\eps>0}$. Let $\rho_{k} \in C\left([0, T], \mathcal{M}\left(\R^d\right)\right)$ be its limit, i.e.
\begin{equation}
\rho_k^{\eps}\rightarrow \rho_{k}
\quad\text{ 
in}\quad C([0, T], \mathcal{M}(\R^d))\text{ as } \eps\to0.
\end{equation}
\begin{lemma}\label{vrho}
For each species $k$, the sequence $\left(u_{k}^{\eps}\right)_{\eps>0}$  converges to $\rho_{k}$ in $C\left([0, T], \mathcal{M}\left(\R^d\right)\right)$ up to a subsequence. 
\end{lemma}
\begin{proof}
The lemma can be implied by
$$
\sup _{0 \leq t \leq T} d\left(\rho^{\eps}_k(t), u^{ \eps}_k(t)\right) \rightarrow 0 \quad\text{as} \quad\varepsilon \rightarrow 0.
$$
 To verify this, we notice that for any $t \in[0, T]$ and $f \in B L$, the following equality holds
$$
\begin{aligned}
 \Big|\int_{\mathbb{R}^d} f(x) u_k^{\eps}(t,x) \d x-\int_{\mathbb{R}^d} f(x) \rho_k^{\eps}(t,\d x)\Big| 
=  &\,  \Big|\int_{\mathbb{R}^d} f(x)\left(\rho_k^{\eps}(t) \ast V^{\varepsilon}\right)(x)\d x-\int_{\mathbb{R}^d} f(x) \rho_k^{\eps}(t,\d x)\Big|\\=  &\,  \Big|\int_{\mathbb{R}^d}\Big(\left(f \ast  V^{\varepsilon}\right)(x)-f(x)\Big) \rho_k^{\eps}(t,\d x)\Big| , \\
\end{aligned}
$$
where $V^\eps$ is even. And then it holds
$$
\begin{aligned}
 &\, \Big|\int_{\mathbb{R}^d} f(x) u_k^{\eps}(t,x) \d x-\int_{\R^d} f(x) \rho_k^{\eps}(t,\d x)\Big|\\
\leq  &\,  \int_{\R^d}\Big(\int_{\R^d}\big|f(x+y)-f(y)\big| V^{\varepsilon}(x) \d x\Big) \rho_k^{\eps}(t,\d y) \\
\leq  &\,  \int_{\mathbb{R}^d}|x| V^{\varepsilon}(x) \d x \int_{\mathbb{R}^d} \rho_k^{\eps}(t,\d y) =C\eps ,
\end{aligned}
$$
which implies our lemma.
\end{proof}
By a priori estimate Lemma \ref{apriori}, for any species $k$, the sequence $(u_{k}^{\eps})_{\eps>0}$  is bounded in $L^{\infty}([0,T], L^m(\R^d))$.  Banach-Alaoglu theorem implies that, up to a subsequence, it weakly$^\ast$ converges in $L^{\infty}([0,T], L^m(\R^d))$. 
Thus by  Lemma \ref{vrho}, we get $\rho_k\in L^{\infty}([0,T], L^m(\R^d))$.  Next, we are going to prove the convergence also holds in $L^m([0,T]\times \R^d)$.
\begin{lemma}\label{Lm convergence} 
For each species $k$, up to a subsequence, $\left(u_{k}^{\eps}\right)_{\eps>0}$ strongly converges in $L^m([0,T]\times\R^d)$ to $\rho_{k}$. 
\end{lemma}
We  claim that it suffices to prove the convergence result in $L^1([0,T]\times B_R)$
for any fixed $R>0$, i.e.,

\begin{equation}\label{LBR}
u_{k}^{\eps}\rightarrow \rho_{k}
\text{ strongly in } L^1([0,T]\times B_R) \text{ as } \eps\to0.
\end{equation}
It is true because of the following remark.
\begin{remark}
By \textit{Vitali convergence theorem},  the sequence $(u_{k}^{\eps})_{\eps>0}$ converges in $L^m([0,T]\times \R^d)$ to $\rho_k$ if and only if 
\begin{itemize}
    \item [(i)] the sequence $(u_{k}^{\eps})_{\eps>0}$ converges in the Lebesgue measure on $[0,T]\times \R^d$ to $\rho_k$;
\item[(ii)] the functions $(u_{k}^{\eps})^m$ are uniformly integrable;
\item[(iii)] for every $\eta>0$, there exists a set $E_\eta\in [0,T]\times \R^d$ of finite measure, such that $\iint_{E_\eta^c}|u_{k}^{\eps} |^m<\eta$ for all $\eps$.    
\end{itemize}
\end{remark}
Actually, by Lemma \ref{vrho} and Prokhorov's theorem, for any $k$, $(u_{k}^{\eps})_{\eps\geq0}$ are uniformly tight. Thus for any $\eta_k>0$ there exists $R_{\eta_k}>0$ such that for any $\eps>0$, it holds
$$
\int_0^T\int_{B_{R_{\eta_k}}^c} u_{k}^{\eps}(t, x) \d x \d t \leq \eta_k.
$$
Then for any $\delta>0$, there exists a ball $B_{R_\delta}$ such that for any $\eps>0$
$$
\int_0^T\int_{B_{R_\delta}^c} u_{k}^{\eps}(t, x) \d x \d t \leq \frac{\delta}{4}\quad \text{and}\quad \int_0^T\int_{B_{R_\delta}^c} \rho_{k}(t, x) \d x \d t \leq \frac{\delta}{4}.
$$
And we have
$$
\begin{aligned}
\int_0^T\int_{\R^d}|u_{k}^{\eps}-\rho_{k}|\d x\d t=  &\, \int_0^T\int_{B_{R_\delta}^c}|u_{k}^{\eps}-\rho_{k}|\d x \d t+\int_0^T\int_{B_{R_\delta}}|u_{k}^{\eps}-\rho_{k}|\d x\d t\\\leq &\,  \frac{\delta}{2}+\int_0^T\int_{B_{R_\delta}}|u_{k}^{\eps}-\rho_{k}|\d x\d t.    
\end{aligned}
$$
That is to say if for any $B_R$ the sequence $(u_{k}^{\eps})_{\eps>0}$ converges to $\rho_k$ in $L^1([0,T]\times B_R)$, then it implies the convergence also holds in $L^1([0,T]\times \R^d)$, which further implies the convergence holds in Lebesgue measure. Statement (i) of the remark is satisfied for sure.

By Remark \ref{remark} the sequence $\left(\sum_ka_ku_{k}^{\eps}\right)^{m / 2}$ is bounded in $L^2\left([0, T], H^1\left(\mathbb{R}^d\right)\right)$.
If $\omega\in H^1(\R^d)$, 
then we have $\omega^2\in W^{1,1}(\R^d)\subset L^{d/(d-1)}(\R^d)$. So $\left(\sum_ka_ku_{k}^{\eps}\right)^{m}$ is bounded in 
\\$L^2\left([0, T], L^{d/(d-1)}(\R^d)\right)$. By the positivity of each $a_k$ and $u_{k}^{\eps}$, for any $k$, it holds $$(u_{k}^{\eps})^{m}\in L^2\big([0, T], L^{d/(d-1)}(\R^d)\big)\subset L^{d/(d-1)}\big([0, T]\times\R^d\big), $$
which deduces that the following uniform in $\eps$ bound  holds
 $$ \int_0^T \int_{\R^d} (u_{k}^{\eps}(t,x))^{\frac{m d}{d-1}}\d x \d t \leq C_k.$$
Since for any set $A\subset[0,T]\times \R^d$ with the characteristic function $\chi_A$ and volume $|A|$, we have 
$$
\begin{aligned}
\lim_{|A|\to 0}\sup_{\eps>0}\iint_A(u_{k}^{\eps}(t,x))^m\d x \d t\leq &\,  \lim_{|A|\to 0}\sup_{\eps>0}\big(\int_0^T \int_{\R^d}(\chi_A)^d\d x\big)^\frac{1}{d}\big(\int_0^T \int_{\R^d}(u_{k}^{\eps}(t,x))^\frac{md}{d-1}\d x\d t\big)^\frac{d-1}{d}\\\leq &\,  \lim_{|A|\to 0}|A|^\frac{1}{d}C_k^\frac{d-1}{d}=0,    
\end{aligned}
$$
which verifies (ii).
And we get
$$
\begin{aligned}
 &\,  \int_0^T \int_{B_{R_{\eta_k}}^c} (u_{k}^{\eps}(t, x))^m \d x \d t \\
\leq  &\,  \Big(\int_0^T \int_{B_{R_{\eta_k}}^c} u_{k}^{\eps}(t,x) \d x \d t\Big)^{\frac{m}{1-d+m d}}\Big(\int_0^T \int_{B_{R_{\eta_k}}^c} u_{k}^{\eps}(t, x)^{\frac{m d}{d-1}} \d x \d t\Big)^{\frac{(m-1)(d-1)}{1-d+m d}} \\
\leq  &\,  C_k^\prime \eta_k^{\frac{m}{1-d+m d}},
\end{aligned}
$$
i.e., (iii) has been verified, then our claim follows. 

\medskip

The following discussion  is analogous to that in \cite{FP08}. To prove our strong convergence result \eqref{LBR}, we need to prove for each $k$, the sequence $(u_{k}^{\eps})_{\eps>0}$ is a Cauchy sequence in $L^1([0,T]\times B_R)$. From Lemma \ref{apriori} and Remark \ref{remark}, we can only deduce higher regularity about the sum $\sum_ka_ku_k^\eps$. For single species $u_k^\eps$, we will take advantage of the mild form of the solution.  We firstly introduce a fractional-type Sobolev space $X_\alpha$, for $0<\alpha<1$ :
$$
X_\alpha:=\left\{w \in L^1(\mathbb{R}^d) \Big| \sup _{0<|h| \leq 1} \frac{\|w(\cdot+h)-w(\cdot)\|_{L^1(\mathbb{R}^d)}}{|h|^\alpha}<+\infty\right\}.
$$
One can check that this is a Banach space endowed with the norm
$$
\|w\|_{X_\alpha}:=\|w\|_{L^1(\mathbb{R}^d)}+\sup _{0<|h| \leq 1} \frac{\|w(\cdot+h)-w(\cdot)\|_{L^1(\mathbb{R}^d)}}{|h|^\alpha} .
$$
By the Riesz-Fréchet-Kolmogorov theorem \cite[Theorem 4.26]{B11},  
any bounded subset of $X_\alpha$ is compact in $L^1(\Omega)$ for any bounded domain $\Omega \subset \mathbb{R}^d$. 
\begin{lemma}\label{Xalpha} 
The sequence $(u_{k}^{\eps})_{\eps>0}$ is uniformly bounded in $L^1\left([0, T], X_\alpha\right)$ for any $0<\alpha<1$.
\end{lemma} 
\begin{proof}
We recall that
$$
\p_t u_{k}^{\eps}=\nabla\cdot\big(b_k\rho_k^{\eps}\nabla g^\eps\big)\ast  V^\eps+\sigma  \Delta u_{k}^{\eps}, \quad u_{k}^{\eps} (0)=\rho_{k,0}\ast V^\eps.
$$
Let 
$f^{\eps}_k=\big(b_k\rho_k^{\eps}\nabla g^\eps\big)\ast  V^\eps$, then we can write above equation into mild form as:
$$
\begin{aligned}
u^{\eps}_{k}(t)  &\,  =\Gamma(t) \ast_x  u^{\eps}_{k}(0)+\int_0^t\left(\Gamma(t-s) \ast_x \operatorname{div} f^{\eps}_{k}(s)\right) \d s \\
 &\,  =\Gamma(t) \ast_x  u^{\eps}_{k}(0)+\int_0^t\left(\nabla \Gamma(t-s) \ast_x  f^{\eps}_{k}(s)\right) \d s,
\end{aligned}
$$
where $\Gamma(t, x)$ is the heat kernel given by
$$
\Gamma(t, x):= \begin{cases}\frac{1}{(4\sigma \pi t)^{d / 2}} e^{-\frac{|x|^2}{4\sigma t}}  &\,  \text { for } t>0 \\ \delta_x  &\,  \text { for } t=0.\end{cases}
$$
In \cite{FP08}, for any $0<\alpha<1$
 we know $
\Gamma, \nabla \Gamma \in L^1\left([0, T], X_\alpha\right)
$. 
And by the fact
$$
\Big\|\int_0^t\int_{\R^d}\nabla_x \Gamma(t-s,x-z)f_k^\eps(s,z)\d z\d s\Big\|_{L^1([0,T],X_\alpha)}\leq \big\|\nabla \Gamma\big\|_{L^1([0,T],X_\alpha)}\big\|f\big\|_{L^1([0,T]\times\R^d)},
$$
one need the $L^1$-estimate of $f_k^\eps$ that
$$
\begin{aligned}
\big\|f^{\eps}_{k}\big\|_{L^1([0, T] \times \R^d)}  &\,  \leq \int_0^T \int_{\mathbb{R}^d}b_k\left|\nabla g^{\eps}(t, x)\right| \rho^{\varepsilon}_k(t, \d x) \d t \\
 &\,  \leq\Big(T \int_0^T \int_{\mathbb{R}^d}b_k^2\left|\nabla g^{\varepsilon}(t, x)\right|^2 \rho^{\varepsilon}_k(t,\d  x) \d t\Big)^{1 / 2}.
\end{aligned}
$$
Therefore, the sequence $(u_{k}^{\eps})_{\eps>0}$ is uniformly bounded in $L^1\left([0, T], X_\alpha\right)$.
A priori estimate Lemma \ref{apriori}, for any species $k$ tells us $f^{\eps}_{k}$ is bounded in $L^1\left([0, T] \times \mathbb{R}^d\right)$  uniformly in $\eps$. 
\end{proof}
Fix $0<\alpha<1$, and take some big $s>0$ such that $\mathcal{M}(B_R)\hookrightarrow H^{-s}(B_R)$ continuously. And for any $\delta>0$, there exists a constant $C_\delta$ such that for any smooth function $f$ on $\R^d$, the following inequality holds (see also \cite{FP08})
\begin{equation}\label{bounded LBR}
 \|f\|_{L^1(B_R)}\leq \delta \|f\|_{X_\alpha} +C_\delta \|f\|_{H^{-s}(B_R)}  . 
\end{equation}
Taking $\eps$ and $\eps^{\prime}$, applying \eqref{bounded LBR} to $u_{k}^{\eps}(t)-u_{k}^{\eps^{\prime}}(t)$ and integrating in time, we obtain
$$
\begin{aligned}
\big\|u_{k}^{\eps}-u_{k}^{\eps^{\prime}}\big\|_{L^1\left([0, T] \times B_R\right)}  &\,  \leq \delta\big\|u_{k}^{\eps}-u_{k}^{\eps^{\prime}}\big\|_{L^1\left([0, T], X_\alpha\right)}  +C_\delta\big\|u_{k}^{\eps}-u_{k}^{\eps^{\prime}}\big\|_{L^1\left([0, T], H^{-s}\left(B_R\right)\right)} \\
 &\,  \leq \delta\Big(\big\|u_{k}^{\eps}\big\|_{L^1\left([0, T], X_\alpha\right)}+\big\|u_{k}^{\eps^{\prime}}\big\|_{L^1\left([0, T], X_\alpha\right)}\Big) +C_\delta\big\|u_{k}^{\eps}-u_{k}^{\eps^{\prime}}\big\|_{L^1\left([0, T], H^{-s}\left(B_R\right)\right)} \\
 &\,  \leq C_k\big(\delta+C_\delta\int_0^Td(u_{k}^{\eps}(t),u_{k}^{\eps^{\prime}}(t))\d t\big).
\end{aligned}
$$
By the convergence of $(u_{k}^{\eps})_{\eps>0}$ in $C([0,T],\mathcal{M}(\R^d))$, we have
$$
\limsup_{\eps,\eps^{\prime}\to 0}\big\|u_{k}^{\eps}-u_{k}^{\eps^{\prime}}\big\|_{L^1\left([0, T] \times B_R\right)}\leq C_k\delta,
$$
which implies that $\left(u^{\eps}_{k}\right)_{\varepsilon>0}$ is a Cauchy sequence in $L^1\left([0, T] \times B_R\right)$ by the arbitrariness of $\delta$. Together with claim \eqref{LBR}, Lemma \ref{Lm convergence} has been proved. 

We finally show that for any $k$, the limit $\rho_k$ is a weak solution  of cross-diffusion system \eqref{cross-diffusion1}  with initial data $\rho_{k,0}$ as in Definition \ref{definition}. 

\begin{proposition}\label{pass}
For each $k$ species and any test function $f\in C^1([0,T], C^2_b(\R^d))$,  the limit $\rho_{k}$ satisfies the following equation:
\begin{equation}\label{weakk}
\begin{aligned}
      &\int_{\R^d}f(t,x)\rho_k(t,x)
\d x+  \int_0^t\int_{\R^d}\p_s f(s,x)\rho_k(s,x)\d x\d s=\int_{\R^d} f(0,x)\rho_{k,0}(x)\d x   \\
 &\, +\sigma \int_0^t\int_{\R^d}\Delta f(s,x) \rho_k(s,x)\d x\d s  -\int_0^t\int_{\R^d}b_k\rho_k(s,x)\nabla f(s,x) \cdot\nabla\big(\sum_{l=1}^{n}a_{l} \rho_l(s,x)\big)^{m-1} \d x\d s.
    \end{aligned}
\end{equation}
\end{proposition}

\begin{proof}
In terms of  \eqref{eps cross}, for any $f\in C^1([0,T], C^2_b(\R^d))$, we have

\begin{equation}\label{weakkeps}
\begin{aligned}
\int_{\R^d} f(t,x) \rho_k^{\eps}(t,\d x)&+  \int_0^t\int_{\R^d}\p_s f(s,x)\rho_k^\eps(s,\d x)\d s=   \int_{\R^d}f(0,x) \rho_{k,0}( x)\d x \\&+\sigma  \int_0^t \int_{\R^d} \Delta f(s,x) \rho_k^{\eps}(s,\d x) \d s
   -\int_0^t \int_{\mathbb{R}^d}b_k \nabla f(s,x) \cdot \nabla g^{\eps}( s,x) \rho_k^{\eps}(s,\d x)  \d s .
\end{aligned}
\end{equation}
Given that $\rho_k^{\eps}$ converges to $\rho_k$ in $C([0,T],\mathcal{M}(\R^d))$ as $\eps$ approaches 0, we are able to pass to the limit for the first four terms of \eqref{weakkeps}. For the last term, it has
$$
\begin{aligned}
 &\, \Big|\int_0^t \int_{\mathbb{R}^d} \nabla f(x) \cdot \nabla g^{\eps}(s, x) \rho_k^{\eps}(s,\d x)  \d s- \int_0^t \int_{\R^d} \nabla f(x) \cdot \nabla \big(\sum_la_l\rho_l(s,x)\big)^{m-1} \rho_{k}(s, x)\d x \d s  \Big| \\
\leq &\, \Big|\int_0^t \int_{\mathbb{R}^d} \nabla f(x) \cdot \nabla g^{\eps}( s,x) \rho_k^{\eps}(s,\d x)  \d s - \int_0^t \int_{\R^d} \nabla f(x) \cdot \nabla\big(\sum_la_lu_{l}^{\eps}(s, x)\big)^{m-1}u_{k}^{\eps}(s,x) \d x \d s \Big|\\
 &\, + \Big| \int_0^t \int_{\R^d} \nabla f(x) \cdot \nabla\big(\sum_la_lu_{l}^{\eps}(s, x)\big)^{m-1}u_{k}^{\eps}(s, x) \d x \d s \\ &\, \qquad\qquad-\int_0^t \int_{\R^d} \nabla f(x) \cdot \nabla \big(\sum_la_l\rho_l(s,x)\big)^{m-1} \rho_{k}(s, x)\d x \d s \Big| \\
=: &\,  I_1^\eps+I_2^\eps.
\end{aligned}
$$

For $I_1^\eps$, noticing $V^\eps(x-y)=V^\eps(y-x)$, we have 
$$
\begin{aligned}
I_1^\eps=  &\,  \Big|\int_0^t \int_{\R^d} \int_{\R^d} \nabla f(x) \cdot \nabla\big(\sum_la_lu_{l}^{\eps}(s, y)\big)^{m-1}V^{\eps}(x-y) \rho_k^{\eps}(s, \d x) \d y \d s \\
 &\,  -\int_0^t \int_{\R^d} \int_{\R^d} \nabla f(y) \cdot \nabla\big(\sum_la_lu_{l}^{\eps}(s, y)\big)^{m-1} V^{\eps}(y-x) \rho_k^{\eps}(s, \d x)\d y  \d s \Big| \\
\leq  &\,  \int_0^t \int_{\R^d} \int_{\R^d}\big|\nabla f(x)-\nabla f(y)\big|\Big|\nabla\big(\sum_la_lu_{l}^{\eps}(s,y)\big)^{m-1} \Big| V^{\eps}(x-y) \rho_k^{\eps}(s, \d x) \d y \d s \\
\leq  &\,  \left\|\nabla^2 f\right\|_{L^\infty} \int_0^t \int_{\R^d} \int_{\R^d}\Big|\nabla\big(\sum_la_lu_{l}^{\eps}(s,y)\big)^{m-1} \Big|\big|x-y\big| V^{\eps}(x-y) \rho_k^{\eps}(s,\d x) \d y \d s ;
\end{aligned}
$$
recall that $V\in C_c^\infty(\R^d)$, then 
$$
\sup_{x,y\in \operatorname{supp}V^\eps}|x-y|<C\eps,
$$
which implies that
$$
\begin{aligned}
I_1^\eps\leq   &\, \eps C\left\|\nabla^2 f\right\|_{L^\infty} \int_0^t \int_{\R^d} \Big|\nabla\big(\sum_la_lu_{l}^{\eps}(s,y)\big)^{m-1} \Big|  u_{k}^{\eps}(s,y) \d y \d s.
\end{aligned}
$$
And we have
$$
\begin{aligned}
 &\, \Big\|u_{k}^{\eps} \nabla \big(\sum_la_lu_{l}^{\eps}\big)^{m-1}\Big\|_{L^1\left([0, T] \times \R^d\right)} \\= &\,  (m-1)  \Big\|u_{k}^{\eps}(\sum_la_lu_{l}^{\eps}\big)^{m-2} \nabla \big(\sum_la_lu_{l}^{\eps}\big)\Big\|_{L^1\left([0, T] \times \R^d\right)}\\
\leq &\,  (m-1)\Big\|\big(\sum_la_lu_{l}^{\eps}\big)^{m / 2-1}\nabla \big(\sum_la_lu_{l}^{\eps}\big)\Big\|_{L^2\left([0, T] \times \R^d\right)}\Big\|u_{k}^{\eps}\big(\sum_la_lu_{l}^{\eps}\big)^{m / 2-1}\Big\|_{L^2\left([0, T] \times \R^d\right)} \\
 \leq &\, \frac{m-1}{\min_k a_k}\Big\|\big(\sum_la_lu_{l}^{\eps}\big)^{m-2}\Big|\nabla \big(\sum_la_lu_{l}^{\eps}\big)\Big|^2\Big\|_{L^1\left([0, T] \times \R^d\right)}^{1 / 2}\Big\|\sum_la_lu_{l}^{\eps}\Big\|_{L^m\left([0, T] \times \mathbb{R}^d\right)}^{m / 2},
\end{aligned}
$$
which is bounded according to a priori estimate Lemma \ref{apriori}. Thus we obtain $I_1^\eps$ goes to $0$ as $\eps\to0$. 

For $I_2^\eps$, we have
$$
\begin{aligned}
I_2^\eps\leq &\, \| \nabla f\|_{L^\infty}\int_0^t \int_{\R^d} \Big|\nabla \big(\sum_la_lu_{l}^{\eps}(s,x)\big)^{m-1} u_{k}^{\eps}(s, x)-\nabla \big(\sum_la_l\rho_{l}(s,x)\big)^{m-1} \rho_{k}(s, x)\Big| \d x\d s.
\end{aligned}
$$
To prove $I_2^\eps\to 0$ ands $\eps\to0$, we need to prove
$$
\nabla \big(\sum_la_lu_{l}^{\eps}(s,x)\big)^{m-1} u_{k}^{\eps}( s,x)\rightarrow \nabla \big(\sum_la_l\rho_{l}(s,x)\big)^{m-1} \rho_{k}(s, x)\quad\text{strongly in}\quad L^1([0,T]\times\R^d).
$$
We possess that for any $k$
\begin{equation}\label{strongm}
u_{k}^{\eps}\rightarrow \rho_{k}
\quad\text{strongly in}\quad L^m([0,T]\times\R^d) \text{ as } \eps\to0.
\end{equation}
So we only need to verify that 
\begin{equation}\label{weak}
\nabla \big(\sum_la_lu_{l}^{\eps}\big)^{m-1}\rightharpoonup \nabla\big(\sum_la_l\rho_{l}\big)^{m-1} 
\quad\text{weakly in}\quad L^{\frac{m}{m-1}}([0,T]\times\R^d) \text{ as } \eps\to0.
\end{equation}

\begin{lemma}\label{Sobolev}
 Assume $m\geq 2$, for any function $h$ such that $h\in L^{m}([0,T]\times\R^d)$ and $h^{m/2}\in L^{2}([0,T], H^1(\R^d))$, then $h^{m-1}\in L^{\frac{m}{m-1}}([0,T], W^{1,\frac{m}{m-1}}(\R^d))$.
\end{lemma}
\begin{proof} 
For $m=2$, the lemma holds trivially; for $m>2$, 
$h^{m-1}\in L^{\frac{m}{m-1}}([0,T]\times\R^d)$ is apparent, we need to prove $\nabla h^{m-1}\in L^{\frac{m}{m-1}}([0,T]\times\R^d)$ as follows
$$
\begin{aligned}
\int_0^T \int_{\R^d} \big|\nabla h^{m-1}\big| &\, ^{\frac{m}{m-1}}\d x \d s=C_m\int_0^T \int_{\R^d} \big|h^{\frac{m}{2}-1}\nabla h^{\frac{m}{2}}\big|^{\frac{m}{m-1}}\d x \d s\\
\leq &\, C_m\Big(\int_0^T \int_{\R^d}h^{(\frac{m}{2}-1)\frac{m}{m-1}\frac{2m-2}{m-2}}\d x \d s\Big)^{\frac{m-2}{2m-2}}\Big(\int_0^T \int_{\R^d}\big|\nabla h^{\frac{m}{2}}\big|^{\frac{m}{m-1}\frac{2m-2}{m}}\d x \d s\Big)^{\frac{m}{2m-2}}\\
= &\, C_m\Big(\int_0^T \int_{\R^d}h^{m}\d x \d s\Big)^{\frac{m-2}{2m-2}}\Big(\int_0^T \int_{\R^d}\big|\nabla h^{\frac{m}{2}}\big|^{2}\d x \d s\Big)^{\frac{m}{2m-2}}<\infty.
\end{aligned}
$$
\end{proof}
By Remark \ref{remark}, we have $u_{k}^{\eps}\in L^{m}([0,T]\times\R^d)$ and $(\sum_ka_ku_{k}^{\eps})^{m/2}\in L^{2}([0,T], H^1(\R^d))$, which implies the limit $\rho_k\in L^{m}([0,T]\times \R^d)$ and $(\sum_ka_k\rho_k)^{m/2}\in L^{2}([0,T], H^1(\R^d))$. 
Applying the lemma above, we have $$ \big(\sum_la_lu_{l}^{\eps}\big)^{m-1} \text{ and } \big(\sum_la_l\rho_{l}\big)^{m-1} \in L^{\frac{m}{m-1}}([0,T], W^{1,\frac{m}{m-1}}(\R^d)).$$ 
For any  test function $h\in L^m([0,T],W^{1,m}(\R^d))$, it holds 
$$
\begin{aligned}
 &\, \Big|\int_0^t \int_{\R^d} h(x,s)\Big(\nabla \big(\sum_la_lu_{l}^{\eps}(s,x)\big)^{m-1}-\nabla\big(\sum_la_l\rho_l(s,x)\big)^{m-1} \Big ) \d x\d s\Big|\\
=  &\, \Big|\int_0^t \int_{\R^d} \nabla h(x,s)\Big( \big(\sum_la_lu_{l}^{\eps}(s,x)\big)^{m-1}-\big(\sum_la_l\rho_l(s,x)\big)^{m-1} \Big ) \d x\d s\Big|\\
\leq  &\, \|\nabla  h\|_{ L^m([0,T]\times\R^d)} \Big(\int_0^t \int_{\R^d} \Big| \big(\sum_la_lu_{l}^{\eps}(s,x)\big)^{m-1}-\big(\sum_la_l\rho_l(s,x)\big)^{m-1} \Big |^{\frac{m}{m-1}} \d x\d s\Big)^{\frac{m-1}{m}},
\end{aligned}
$$
which goes to $0$ as $\eps\to0$ thanks to
$$
 \big(\sum_la_lu_{l}^{\eps}\big)^{m-1}\rightarrow \big(\sum_la_l\rho_{l}\big)^{m-1} 
\quad\text{strongly in}\quad L^{\frac{m}{m-1}}([0,T]\times\R^d).
$$
Now for any $h\in L^m([0,T]\times\R^d)$, it can be approximated by a sequence \\$(h_n)_{n\geq 1}\in L^m([0,T],W^{1,m}(\R^d))$.  Thus 
$$
\int_0^t \int_{\R^d}h\nabla \big(\sum_la_lu_{l}^{\eps}(s,x)\big)^{m-1}\d x\d s\rightarrow \int_0^t \int_{\R^d}h\nabla\big(\sum_la_l\rho_{l}(s,x)\big)^{m-1} \d x\d s, \text{ as } \eps\to0,
$$ which is due to both  $\big(\sum_lu_{l}^{\eps}\big)^{m-1}$ and $\big(\sum_l\rho_{l}\big)^{m-1}$ are bounded in $L^{\frac{m}{m-1}}([0,T], W^{1,\frac{m}{m-1}}(\R^d))$ and the dominated convergence theorem enables to let $h_n\to h$. Hence, the weak convergence as expressed in \eqref{weak} is achieved as desired. Combining \eqref{strongm} and \eqref{weak}, we obtain $I_2^\eps\to0$.

Thus we finish the proof of Proposition \ref{pass}.
\end{proof}

In conclusion, $\rho=(\rho_1,\ldots,\rho_n)$ is a weak solution of the cross-diffusion system \eqref{cross-diffusion1}.

\section{Proof of Theorem \ref{uniqueness}}\label{proof of uniqueness}

In this section, we will prove Theorem \ref{uniqueness} under the assumption that mobilities $b_1,\ldots,b_k$ are the same. Namely, the cross-diffusion equation \eqref{cross-diffusion1} reduces to the following system 
\begin{equation}\label{eq same b}
\begin{cases}
\p_t \rho_k-b\nabla\cdot \big(\rho_{k}\nabla P(\rho)\big)=\sigma\Delta \rho_{k},\\
\rho_{k}(0)=\rho_{k,0}, \quad \rho_{k,0}\in  L^1\cap L^\infty(\R^d),
\end{cases}k=1,2,\ldots,n,
\end{equation}
with the pressure defined as
$$
P(\rho)=\Big(\sum_{l=1}^{n}a_l\rho_l\Big)^{m-1},\quad m\geq 2.
$$ 
If we sum up the cross-diffusion system with weights $a_k$, then it holds
$$
\p_t( \sum_la_l\rho_l)=b\nabla\cdot \big(\sum_la_l\rho_l\nabla P(\rho)\big)+\sigma\Delta \sum_la_l\rho_l.
$$
Let 
$
u=\sum_la_l\rho_l$ with $ u_0=\sum_la_l\rho_{l,0}$,
which satisfies formally that
\begin{equation}\label{u}
\p_t u=b\nabla\cdot(u\nabla u^{m-1})+\sigma\Delta u,
\end{equation}
where $u_0\in L^1\cap L^\infty(\R^d)$. Weak solutions of \eqref{u} are understood in sense of the following definition. 
\begin{definition}\label{def weak u}
A weak solution of \eqref{u} on the time interval $[0,T]$ satisfies $u\in L^{m}([0,T]\times\R^d)$ and $u^{m-1}\in  L^{\frac{m}{m-1}}(0,T;W^{1,\frac{m}{m-1}}(\R^d))$. And
    for any $f\in C^1([0,T], C^2_b(\R^d))$ and almost any $t\in[0,T]$, it holds
\begin{equation}\label{eq weak u}
\begin{aligned}
      &\int_{\R^d}f(t,x)u(t,x)
\d x+  \int_0^t\int_{\R^d}u(s,x)\p_s f(s,x)\d x\d s=\int_{\R^d} f(0,x)u_{0}(x)\d x   \\
 &\, +\sigma \int_0^t\int_{\R^d}\Delta f(s, x) u(s,x)\d x\d s  -b\int_0^t\int_{\R^d}u(s,x)\nabla f(s,x) \cdot\nabla u^{m-1}(s,x) \d x\d s.
    \end{aligned}     
\end{equation}  
\end{definition}

Given $\rho_{k,0}\in L^1\cap L^\infty(\R^d)$ for $k=1,\ldots,n$, if the weak solution of \eqref{u} defined as Definition \ref{def weak u} is unique with initial data $ u_0=\sum_ka_k\rho_{k,0}$, then the pressure of the cross-diffusion \eqref{eq same b} satisfied by $P(\rho)=u^{m-1}$ is uniquely determined with certain regularity. Thus, each $\rho_k$ satisfies the linear Fokker-Planck equation
$$
\p_t \rho_k=b\nabla\cdot(\rho_k \nabla u^{m-1})+\sigma \Delta \rho_k,
$$
which has a unique weak solution $\rho_k\in L^m([0,T]\times\R^d)$ in sense of Definition \ref{definition} \cite[Theorem 9.3.6]{bogachev2022fokker}. Since a priori estimate of \eqref{u} tells us
$$
\frac{\|u(T)\|_{L^m}}{m}+b\int_0^T\int_{\R^d}u|\nabla u^{m-1}|^2\d x\d t+\sigma(m-1)\int_0^T\int_{\R^d} u^{m-2}|\nabla u|^2 \d x\d t=\frac{\|u(0)\|_{L^m}}{m},
$$
which can be made rigorous by regularising first and passing to the limit then. It implies that, for any $k$, the vector field $\nabla u^{m-1}\in L^2([0,T]\times\R^d, \rho_k\d x\d t)$, i.e.,
$$
\int_0^T\int_{\R^d}\rho_k|\nabla u^{m-1}|^2\d x\d t\leq \int_0^T\int_{\R^d}u|\nabla u^{m-1}|^2\d x\d t<\infty.
$$
It also implies $\nabla u^{m-1}\in L^1([0,T]\times\R^d, \rho_k\d x\d t)$ as follows
$$
\begin{aligned}
&\int_0^T\int_{\R^d}\rho_k|\nabla u^{m-1}|\d x\d t\leq \int_0^T\int_{\R^d}u|\nabla u^{m-1}|\d x\d t\\\leq&\,\|u\|_{L^1([0,T]\times\R^d)}^{\frac{1}{2}}\Big(\int_0^T\int_{\R^d}u|\nabla u^{m-1}|^2\d x\d t\Big)^{\frac{1}{2}}<\infty,
\end{aligned}
$$ which verifies the condition needed in \cite[Theorem 9.3.6]{bogachev2022fokker}.

Therefore, to show Theorem \ref{uniqueness}, it is sufficient to prove the following proposition, which is well-known \cite{FP08,Va} but we adapt it here to our regularity setting.

\begin{proposition}\label{prop unique u}
The weak solution of \eqref{u} defined as Definition \ref{def weak u} is unique.   
\end{proposition}

\begin{proof}
Without loss of the generality, we assume $b=\sigma=1$ in the proof. The first step is to show $u\in L^{m+1}([0,T]\times \R^d)$. Recall the smooth mollifier $V^\eps$ defined as in \eqref{eps cross} and let $$u_\eps(t,\cdot)=u(t,\cdot)\ast V^\eps,$$ then it holds in sense of weak form \eqref{eq weak u} that 
$$
\p_t u_\eps=\nabla\cdot \big((u\nabla u^{m-1})\ast V^\eps\big)+\Delta u\ast V^\eps.
$$
We take smooth test function
$$
f(x,t)=\int_t^T(\frac{m-1}{m}u^m+u)\ast V^\eps \d s
$$
with
$$
\nabla f(x,t)=\int_t^T(u\nabla u^{m-1}+\nabla u)\ast V^\eps \d s.
$$
We plug in $f$ into \eqref{eq weak u}. The right-hand side of the equality follows 
$$
\begin{aligned}
  &\int_{\R^d} f(0,x)u_\eps(0,x)\d x    +\int_0^T\int_{\R^d} u_\eps\Delta f\d x\d t  -\int_0^T\int_{\R^d}\nabla f \cdot\big((u\nabla u^{m-1})\ast V^\eps\big)\d x\d t\\
  =&\int_{\R^d} u_\eps(0)\int_0^T\big(\frac{m-1}{m}u^m+u\big)\ast V^\eps\d t \\&-\int_0^T\int_{\R^d}\big((u\nabla u^{m-1}+\nabla u)\ast V^\eps\big)  \Big(\int_t^T\big((u\nabla u^{m-1}+\nabla u)\ast V^\eps\big) \d s\Big) \d x\d t\\=&\int_{\R^d} u_\eps(0)\int_0^T\big(\frac{m-1}{m}u^m+u\big)\ast V^\eps\d t -\frac{1}{2}\int_{\R^d}\Big|\int_0^T\big((u\nabla u^{m-1}+\nabla u)\ast V^\eps\big) \d t \Big|^2 \d x,
\end{aligned}
$$
where we used the integral by parts once for the second order term and the last step comes from the Fubini's theorem. While the left-hand side of \eqref{eq weak u} yields
$$
\begin{aligned}
&\int_{\R^d}f(T,x)u_\eps(T,x)
\d x+  \int_0^T\int_{\R^d}u_\eps(t,x)\p_t f(t,x)\d x\d t\\
=&-\int_0^T\int_{\R^d} u_\eps  \p_t\Big(\int_t^T(\frac{m-1}{m}u^m+u)\ast V^\eps \d s\Big) \d x\d t\\
=&\int_0^T\int_{\R^d} u_\eps  (\frac{m-1}{m}u^m+u)\ast V^\eps  \d x\d t.
\end{aligned}
$$
Combining both identities above with the fact $u_\eps^m\leq u^m\ast V^\eps$ due to Jensen's inequality, we get that
$$
\begin{aligned}
&\frac{m-1}{m}\int_0^T\int_{\R^d} (u_\eps )^{m+1}   \d x\d t +\int_0^T\int_{\R^d} (u_\eps )^{2}   \d x\d t +\frac{1}{2}\int_{\R^d}\Big|\int_0^T\big((u\nabla u^{m-1}+\nabla u)\ast V^\eps\big) \d t \Big|^2 \d x\\
\leq & \int_{\R^d} u_\eps(0)\int_0^T\big(\frac{m-1}{m}u^m+u\big)\ast V^\eps\d t\d x\\
\leq &\|u_\eps(0)\|_{L^\infty}\int_{\R^d}\int_0^T\big(\frac{m-1}{m}u^m+u\big)\ast V^\eps\d t\d x\\
\leq &\|u_\eps(0)\|_{L^\infty}\big(\frac{m-1}{m}\|u\|_{L^m([0,T]\times\R^d)}+\|u\|_{L^1([0,T]\times\R^d)}\big).
\end{aligned}
$$
Taking the limit as $\eps\to0$, we obtain $u\in L^{m+1}(\R^d\times [0,T])$.

Similar as the argument above,  we assume that two solutions $u$ and $\bar u$ of \eqref{u} have the same initial data i.e., $u(0)=\bar u(0)$. Their regularised versions $u_\eps:=u\ast V^\eps$ and $\bar u_\eps:=\bar u\ast V^\eps$ satisfy the following equation
\begin{equation}\label{difference}
\p_t (u_\eps-\bar u_\eps)=\nabla\cdot \big((u\nabla u^{m-1}-\bar u\nabla\bar u^{m-1})\ast V^\eps\big)+\Delta (u_\eps-\bar u_\eps).
\end{equation}
We take test function 
$$
f=\int_t^T(\frac{m-1}{m}u^m+u-\frac{m-1}{m}\bar u^m-\bar u)\ast V^\eps \d s,
$$
and plug in it into the weak form of \eqref{difference} as follows
\begin{equation}\label{eq weak differ}
\begin{aligned}
&\int_{\R^d}f(T,x)(u_\eps(T,x)-\bar u_\eps(T,x))
\d x+  \int_0^T\int_{\R^d}(u_\eps(t,x)-\bar u_\eps(t,x))\p_t f(t,x)\d x\d t\\=&\,\int_{\R^d} f(0,x)(u_\eps(0)-\bar u_\eps(0)) \d x    +\int_0^T\int_{\R^d} (u_\eps-\bar u_\eps)\Delta f\d x\d t \\& -\int_0^T\int_{\R^d}\nabla f \cdot\big((u\nabla u^{m-1})\ast V^\eps-(\bar u\nabla \bar u^{m-1})\ast V^\eps\big)\d x\d t.    
\end{aligned}
\end{equation}  The right-hand side reads as
$$
\begin{aligned}
&\int_{\R^d} f(0,x)(u_\eps(0)-\bar u_\eps(0)) \d x    +\int_0^T\int_{\R^d} (u_\eps-\bar u_\eps)\Delta f\d x\d t \\& -\int_0^T\int_{\R^d}\nabla f \cdot\big((u\nabla u^{m-1})\ast V^\eps-(\bar u\nabla \bar u^{m-1})\ast V^\eps\big)\d x\d t\\
  = &-\int_0^T\int_{\R^d}\big((u\nabla u^{m-1}+\nabla u-\bar u\nabla\bar u^{m-1}-\nabla\bar u)\ast V^\eps\big)  \\&\times\Big(\int_t^T(u\nabla u^{m-1}+\nabla u-\bar u\nabla\bar u^{m-1}-\nabla\bar u)\ast V^\eps\big) \d s\Big) \d x\d t\\
  = &-\frac{1}{2}\int_{\R^d}\Big|\int_0^T\big((u\nabla u^{m-1}+\nabla u-\bar u\nabla\bar u^{m-1}-\nabla\bar u)\ast V^\eps\big) \d t \Big|^2 \d x.
\end{aligned}
$$
And the left-hand side of the identity \eqref{eq weak differ} yields
$$
\begin{aligned}
&\int_{\R^d}f(T,x)(u_\eps(T,x)-\bar u_\eps(T,x))
\d x+  \int_0^T\int_{\R^d}(u_\eps(t,x)-\bar u_\eps(t,x))\p_t f(t,x)\d x\d t\\
=&\int_0^T\int_{\R^d} (u_\eps-\bar u_\eps) (\frac{m-1}{m}u^m+u-\frac{m-1}{m}\bar u^m-\bar u)\ast V^\eps  \d x\d t.
\end{aligned}
$$
Namely, we have 
$$
\begin{aligned}
&\int_0^T\int_{\R^d} (u_\eps-\bar u_\eps) (\frac{m-1}{m}u^m+u-\frac{m-1}{m}\bar u^m-\bar u)\ast V^\eps  \d x\d t\\=& -\frac{1}{2}\int_{\R^d}\Big|\int_0^T\big((u\nabla u^{m-1}+\nabla u-\bar u\nabla\bar u^{m-1}-\nabla\bar u)\ast V^\eps\big) \d t \Big|^2 \d x   \leq0,
\end{aligned}
$$
where we take the limit $\eps\to0$ and infer that 
$$
\begin{aligned}
&\int_0^T\int_{\R^d} (u-\bar u) (\frac{m-1}{m}u^m+u-\frac{m-1}{m}\bar u^m-\bar u) \d x\d t \leq0.
\end{aligned}
$$
But convexity of the function $\frac{m-1}{m}u^m+u$ implies 
$$
 (u-\bar u) (\frac{m-1}{m}u^m+u-\frac{m-1}{m}\bar u^m-\bar u)\geq 0
$$ holds everywhere,
which is also integrable because $u\in L^{m+1}([0,T]\times \R^d)$.
Therefore, it follows $u=\bar u$ almost everywhere.
    
\end{proof}

\section*{Acknowledgments}
The research of JAC was supported by the Advanced Grant Nonlocal\--CPD (Non\-local PDEs for Complex Particle Dynamics: Phase Transitions, Patterns and Synchronization) of the European Research Council Executive Agency (ERC) under the European Union’s Horizon 2020 research and innovation programme (grant agreement No. 883363).
JAC was also partially supported by EPSRC grant numbers EP/T022132/1 and EP/V051121/1.

\bibliography{reference}

\begin{thebibliography}{10}

\bibitem{BGH85}
M.~Bertsch, M.~E. Gurtin, D.~Hilhorst, and L.~Peletier.
\newblock On interacting populations that disperse to avoid crowding: preservation of segregation.
\newblock {\em J. Math. Biol.}, 23(1):1--13, 1985.

\bibitem{bogachev2022fokker}
V.~I. Bogachev, N.~V. Krylov, M.~R{\"o}ckner, and S.~V. Shaposhnikov.
\newblock {\em Fokker--Planck--Kolmogorov Equations}, volume 207.
\newblock American Mathematical Society, 2022.

\bibitem{B11}
H.~Br{\'e}zis.
\newblock {\em Functional analysis, Sobolev spaces and partial differential equations}, volume~2.
\newblock Springer, 2011.

\bibitem{BCM07}
M.~Burger, V.~Capasso, and D.~Morale.
\newblock On an aggregation model with long and short range interactions.
\newblock {\em Nonlinear Analysis: Real World Applications}, 8(3):939--958, 2007.

\bibitem{BCPS20}
M.~Burger, J.~A. Carrillo, J.-F. Pietschmann, and M.~Schmidtchen.
\newblock Segregation effects and gap formation in cross-diffusion models.
\newblock {\em Interfaces Free Bound.}, 22(2):175--203, 2020.

\bibitem{BE23}
M.~Burger and A.~Esposito.
\newblock Porous medium equation and cross-diffusion systems as limit of nonlocal interaction.
\newblock {\em Nonlinear Analysis}, 235:113347, 2023.

\bibitem{CC06}
V.~Calvez and J.~A. Carrillo.
\newblock Volume effects in the {K}eller-{S}egel model: energy estimates preventing blow-up.
\newblock {\em J. Math. Pures Appl. (9)}, 86(2):155--175, 2006.

\bibitem{Patacchini_blob19}
J.~A. Carrillo, K.~Craig, and F.~S. Patacchini.
\newblock A blob method for diffusion.
\newblock {\em Calc. Var. Partial Differential Equations}, 58(2):Paper No. 53, 53, 2019.

\bibitem{CEW23}
J.~A. Carrillo, A.~Esposito, and J.~S.-H. Wu.
\newblock Nonlocal approximation of nonlinear diffusion equations.
\newblock {\em arXiv preprint arXiv:2302.08248}, 2023.

\bibitem{CFSS18}
J.~A. Carrillo, S.~Fagioli, F.~Santambrogio, and M.~Schmidtchen.
\newblock Splitting schemes and segregation in reaction cross-diffusion systems.
\newblock {\em SIAM Journal on Mathematical Analysis}, 50(5):5695--5718, 2018.

\bibitem{CHS18}
J.~A. Carrillo, Y.~Huang, and M.~Schmidtchen.
\newblock Zoology of a nonlocal cross-diffusion model for two species.
\newblock {\em SIAM Journal on Applied Mathematics}, 78(2):1078--1104, 2018.

\bibitem{CMSTT19}
J.~A. Carrillo, H.~Murakawa, M.~Sato, H.~Togashi, and O.~Trush.
\newblock A population dynamics model of cell-cell adhesion incorporating population pressure and density saturation.
\newblock {\em Journal of theoretical biology}, 474:14--24, 2019.

\bibitem{CDHJ21}
L.~Chen, E.~S. Daus, A.~Holzinger, and A.~J{\"u}ngel.
\newblock Rigorous derivation of population cross-diffusion systems from moderately interacting particle systems.
\newblock {\em Journal of Nonlinear Science}, 31:1--38, 2021.

\bibitem{CDJ19}
L.~Chen, E.~S. Daus, and A.~J{\"u}ngel.
\newblock Rigorous mean-field limit and cross-diffusion.
\newblock {\em Zeitschrift f{\"u}r angewandte Mathematik und Physik}, 70:1--21, 2019.

\bibitem{DDD19}
E.~S. Daus, L.~Desvillettes, and H.~Dietert.
\newblock About the entropic structure of detailed balanced multi-species cross-diffusion equations.
\newblock {\em Journal of Differential Equations}, 266(7):3861--3882, 2019.

\bibitem{DHPP24}
M.~Doumic, S.~Hecht, B.~Perthame, and D.~Peurichard.
\newblock Multispecies cross-diffusions: from a nonlocal mean-field to a porous medium system without self-diffusion.
\newblock {\em Journal of Differential Equations}, 389:228--256, 2024.

\bibitem{DHJ23}
P.-{\'E}. Druet, K.~Hopf, and A.~J{\"u}ngel.
\newblock Hyperbolic--parabolic normal form and local classical solutions for cross-diffusion systems with incomplete diffusion.
\newblock {\em Communications in Partial Differential Equations}, pages 1--32, 2023.

\bibitem{DJ20}
P.-E. Druet and A.~Jüngel.
\newblock Analysis of cross-diffusion systems for fluid mixtures driven by a pressure gradient.
\newblock {\em SIAM Journal on Mathematical Analysis}, 52(2):2179--2197, 2020.

\bibitem{DB15}
L.~Dyson and R.~E. Baker.
\newblock The importance of volume exclusion in modelling cellular migration.
\newblock {\em Journal of mathematical biology}, 71:691--711, 2015.

\bibitem{FBC23}
C.~Falc{\'o}, R.~E. Baker, and J.~A. Carrillo.
\newblock A local continuum model of cell-cell adhesion.
\newblock {\em SIAM Journal on Applied Mathematics}, pages S17--S42, 2023.

\bibitem{FP08}
A.~Figalli and R.~Philipowski.
\newblock Convergence to the viscous porous medium equation and propagation of chaos.
\newblock {\em ALEA Lat. Am. J. Probab. Math. Stat}, 4:185--203, 2008.

\bibitem{FLO19}
F.~Flandoli, M.~Leimbach, and C.~Olivera.
\newblock Uniform convergence of proliferating particles to the fkpp equation.
\newblock {\em Journal of Mathematical Analysis and Applications}, 473(1):27--52, 2019.

\bibitem{FM15}
J.~Fontbona and S.~M{\'e}l{\'e}ard.
\newblock Non local lotka-volterra system with cross-diffusion in an heterogeneous medium.
\newblock {\em Journal of mathematical biology}, 70:829--854, 2015.

\bibitem{G16}
F.~Golse.
\newblock On the dynamics of large particle systems in the mean field limit.
\newblock {\em Macroscopic and large scale phenomena: coarse graining, mean field limits and ergodicity}, pages 1--144, 2016.

\bibitem{GL23}
S.~Guo and D.~Luo.
\newblock Scaling limit of moderately interacting particle systems with singular interaction and environmental noise.
\newblock {\em The Annals of Applied Probability}, 33(3):2066--2102, 2023.

\bibitem{IRS12}
K.~Ichikawa, M.~Rouzimaimaiti, and T.~Suzuki.
\newblock Reaction diffusion equation with non-local term arises as a mean field limit of the master equation.
\newblock {\em Discrete and Continuous Dynamical Systems-S}, 5(1):115--126, 2011.

\bibitem{JW17}
P.-E. Jabin and Z.~Wang.
\newblock Mean field limit for stochastic particle systems.
\newblock {\em Active Particles, Volume 1: Advances in Theory, Models, and Applications}, pages 379--402, 2017.

\bibitem{JM98}
B.~Jourdain and S.~M{\'e}l{\'e}ard.
\newblock Propagation of chaos and fluctuations for a moderate model with smooth initial data.
\newblock In {\em Annales de l'Institut Henri Poincare (B) Probability and Statistics}, volume~34, pages 727--766. Elsevier, 1998.

\bibitem{J16}
A.~J{\"u}ngel.
\newblock {\em Entropy methods for diffusive partial differential equations}, volume 804.
\newblock Springer, 2016.

\bibitem{LLP17}
T.~Lorenzi, A.~Lorz, and B.~Perthame.
\newblock On interfaces between cell populations with different mobilities.
\newblock {\em Kinet. Relat. Models}, 10(1):299--311, 2017.

\bibitem{M20}
A.~Moussa.
\newblock From nonlocal to classical shigesada--kawasaki--teramoto systems: Triangular case with bounded coefficients.
\newblock {\em SIAM Journal on Mathematical Analysis}, 52(1):42--64, 2020.

\bibitem{O85}
K.~Oelschl{\"a}ger.
\newblock A law of large numbers for moderately interacting diffusion processes.
\newblock {\em Zeitschrift f{\"u}r Wahrscheinlichkeitstheorie und verwandte Gebiete}, 69(2):279--322, 1985.

\bibitem{O90}
K.~Oelschl{\"a}ger.
\newblock Large systems of interacting particles and the porous medium equation.
\newblock {\em Journal of differential equations}, 88(2):294--346, 1990.

\bibitem{O01}
K.~Oelschl{\"a}ger.
\newblock A sequence of integro-differential equations approximating a viscous porous medium equation.
\newblock {\em Zeitschrift f{\"u}r Analysis und ihre Anwendungen}, 20(1):55--91, 2001.

\bibitem{P07}
R.~Philipowski.
\newblock Interacting diffusions approximating the porous medium equation and propagation of chaos.
\newblock {\em Stochastic processes and their applications}, 117(4):526--538, 2007.

\bibitem{PR07}
C.~Pr{\'e}v{\^o}t and M.~R{\"o}ckner.
\newblock {\em A concise course on stochastic partial differential equations}, volume 1905.
\newblock Springer, 2007.

\bibitem{S18}
I.~Seo.
\newblock Scaling limit of two-component interacting brownian motions.
\newblock {\em The Annals of Probability}, 46(4):2038--2063, 2018.

\bibitem{S00}
A.~Stevens.
\newblock The derivation of chemotaxis equations as limit dynamics of moderately interacting stochastic many-particle systems.
\newblock {\em SIAM Journal on Applied Mathematics}, 61(1):183--212, 2000.

\bibitem{S91}
A.-S. Sznitman.
\newblock Topics in propagation of chaos.
\newblock {\em Lecture notes in mathematics}, pages 165--251, 1991.

\bibitem{Va}
J.~L. V\'azquez.
\newblock {\em The porous medium equation}.
\newblock Oxford Mathematical Monographs. The Clarendon Press, Oxford University Press, Oxford, 2007.
\newblock Mathematical theory.

\end{thebibliography}
\bibliographystyle{abbrv}

\end{document}